\documentclass[preprint,12pt]{elsarticle}

\makeatletter
\def\ps@pprintTitle{%
 \let\@oddhead\@empty
 \let\@evenhead\@empty
 \def\@oddfoot{\centerline{\thepage}}%
 \let\@evenfoot\@oddfoot}
\makeatother

\usepackage{graphicx}
\usepackage{amssymb}
\usepackage{longtable}
\usepackage{tikz}
\usepackage{amsmath}
\usepackage{amsthm}

\theoremstyle{plain}
\newtheorem{theorem}{Theorem}[section] 
\newtheorem{lemma}[theorem]{Lemma}

\newtheorem{corollary}[theorem]{Corollary}
\newtheorem{thm}[theorem]{Theorem}
\newtheorem{lem}[theorem]{Lemma}

\theoremstyle{definition}

\newtheorem{defn}[theorem]{Definition}

\theoremstyle{remark}
 
\newtheorem{remark}{Remark}
\newtheorem{conjecture}{Conjecture}

\newcommand \EQ {\overset{n}{\equiv}}

\journal{arXiv}

\begin{document}

\begin{frontmatter}
\title{Classification of edge-transitive propeller graphs}
\author{Matthew C. Sterns}
\ead{mcs279@nau.edu}
\address{Texas, United States}

\begin{abstract}
In this paper, we introduce a family of tetravalent graphs called \emph{propeller graphs}, denoted by $Pr_{n}\left(b,c,d\right)$. We then produce three infinite subfamilies and one finite subfamily of arc-transitive propeller graphs, and show that all such graphs are necessarily members of one of these four subfamilies, up to isomorphism. We close the paper with questions for further investigation, as well as a few conjectures.
\end{abstract}


\end{frontmatter}

\section{Introduction}
\label{S:Intro}

All graphs in this paper are assumed to be finite, simple, and connected. An \emph{automorphism} of a graph is a permutation on the vertex set which preserves the edges of the graph. Often an author will use the term \emph{symmetry} to refer to an automorphism, though we will not do so here. An automorphism of a graph is called $\left(m,n\right)$\emph{-semiregular} if it has precisely $m$ vertex-orbits, each of length $n$. A \emph{circulant graph}, or \emph{circulant}, is a graph which admits a $\left(1,n\right)$-semiregular automorphism. \emph{Bicirculant}, \emph{tricirculant}, \emph{tetracirculant}, and \emph{pentacirculant} graphs are defined similarly, admitting $\left(m,n\right)$-semiregular automorphisms for $m=2$, $3$, $4$, or $5$, respectively. For any positive integer $s$, an $s$\emph{-arc} in a graph is a sequence of $s+1$ vertices of the graph, say $v_{0},v_{1},\ldots,v_{s}$, not necessarily all distinct, such that any two consecutive terms are adjacent and any three consecutive terms are distinct. A $1$-arc will simply be called an \emph{arc}. We say that $\Gamma$ is \emph{vertex-transitive}, \emph{edge-transitive}, \emph{arc-transitive}, or $s$\emph{-arc-transitive} provided that $G\le Aut\left(\Gamma\right)$ acts transitively on the sets of vertices, edges, arcs, or $s$-arcs of $\Gamma$, respectively. For a vertex $v$, we use $N(v)$ to denote the neighborhood of $v$, that is, the vertices adjacent to $v$.

In recent years, there has been increasing interest in classifying edge-transitive and arc-transitive graphs relative to their valency and circulancy. All arc-transitive circulant graphs were identified independently by Kovacs \cite{ClassCirc} and Li \cite{LiClassCirc}. For trivalent graphs, the classification of bicirculants was handled by Pisanski \cite{ClassCubicBicirc}, relying partly on the results of Frucht, Graver, and Watkins regarding the generalized Petersen graphs \cite{GroupsGenPet} and on joint work with Marusic \cite{SymmHexMolecGraph}. Kovacs, Kutnar, Marusic, and Wilson completed the classification of arc-transitive trivalent tricirculants in 2012 \cite{ClassCubicTricirc} while Frelih and Kutnar that same year found all arc-transitive trivalent tetracirculant and pentacirculant graphs \cite{ClassCubicTetracircPentacirc}. Meanwhile, Kovacs, Kuzman, Malnic, and Wilson exhausted the families of edge-transitive tetravalent bicirculants \cite{ClassTetraBicirc} which, in a similar fashion to the cubic bicirculants and generalized Petersen graphs, largely depended on the classification of edge-transitive rose window graphs \cite{RoseClass}. A classification of the edge-transitive or arc-transitive tetravalent tricirculants has yet to be completed.

It is our hope that this paper will constitute a step toward completing that classification. In section \ref{S:PropGraphs}, we introduce a construction for a family of tetravalent tricirculant graphs which we call \emph{propeller graphs}. We will present a pair of automorphisms common to all propeller graphs: a $\left(3,n\right)$-semiregular automorphism, and an arc-reversing automorphism which establishes that all edge-transitive propeller graphs are arc-transitive also. A particular property of this family is that all propeller graphs have a girth $g$ of at most $6$, which will prove useful in proving the primary theorem of this paper. In section \ref{S:Families}, we present three infinite subfamilies and one finite subfamily of arc-transitive propeller graphs. Section \ref{S:Classification} describes the organization of the classification proof into arguments by girth. In section \ref{S:g=3}, we show that edge-transitive propeller graphs with $g=3$ must be line graphs of $2$-arc-transitive generalized Petersen graphs, which are well known. In section \ref{S:g=4}, we show that edge-transitive propeller graphs with $g=4$ are the underlying graphs of edge-transitive toroidal maps of type $\{4,4\}$, while the edge-transitive propeller graphs where $g\ge 5$ are handled in section \ref{S:g>4}. In section \ref{S:Further}, we conclude the paper with some open questions of further research and investigation, and make conjectures about the automorphism groups and consistent cycles of propeller graphs.

\section{Propeller Graphs}
\label{S:PropGraphs}

\begin{defn}
Let $n\ge 3$ be an integer and let $0<b,c,d<n$ with $d\ne \frac{n}{2}$. The \emph{propeller graph} $\Pr_{n}\left(b,c,d\right)$ is a tetravalent graph with vertex set
\[\{A_{0},A_{1},A_{2},\ldots,A_{n-1},B_{0},B_{1},B_{2},\ldots,B_{n-1},C_{0},C_{1},C_{2},\ldots,C_{n-1}\}\]
and edge set
\[\{\{A_{i},A_{i+1}\},\{A_{i},B_{i}\},\{B_{i},A_{i+b}\},\{B_{i},C_{i+c}\},\{C_{i},B_{i}\},\{C_{i},C_{i+d}\}\vert i\in \mathbb{Z}_{n}\}.\]
\end{defn}

All arithmetic on indices is presumed to be modulo $n$, unless otherwise specified.

We will sometimes refer to the vertices as $A$\emph{ vertices}, $B$\emph{ vertices}, and $C$\emph{ vertices}. It will also help later on to have names for the types of edges of a given propeller graph. Respective to the order given in the definition of the edge set, a propeller graph has $A$\emph{-wings}, $A$\emph{-flats}, $A$\emph{-blades}, $C$\emph{-blades}, $C$\emph{-flats}, and $C$\emph{-wings}, all of which appear in the graph in equal amounts. We define $A$\emph{-spokes} to be the union of $A$-flats and $A$-blades, and define $C$\emph{-spokes} analogously.

By construction, all propeller graphs admit the $\left(3,n\right)$-semiregular automorphism
\[
\begin{tabular}{rr|l}
\cline{3-3} 
 & $A_{i}\mapsto$ & $A_{i+1}$\tabularnewline
$\rho:$ & $B_{i}\mapsto$ & $B_{i+1}$\tabularnewline
 & $C_{i}\mapsto$ & $C_{i+1}$\tabularnewline
\end{tabular}
\]

\[\rho=\left(A_{0},A_{1},A_{2},\ldots,A_{-1}\right)\left(B_{0},B_{1},B_{2},\ldots,B_{-1}\right)\left(C_{0},C_{1},C_{2},\ldots,C_{-1}\right).\]
which induces a cyclic subgroup of the full automorphism group of the graph.

All propeller graphs also admit the involution
\[
\begin{tabular}{rr|l}
\cline{3-3} 
 & $A_{i}\mapsto$ & $A_{-i}$\tabularnewline
$\mu:$ & $B_{i}\mapsto$ & $B_{-i-b}$\tabularnewline
 & $C_{i}\mapsto$ & $C_{-i-b+c}$\tabularnewline
\end{tabular}
\]
which preserves the set of $A$-wings and the set of $C$-wings, transposes the sets of $A$-flats and $A$-blades, and transposes the sets of $C$-flats and $C$ -blades. There are therefore at most four orbits of edges in a propeller graph: the $A$-wings, the $A$-spokes, the $C$-spokes, and the $C$-wings. Since $\mu\rho$ acts as a reflection of the edge $\{A_{0},A_{1}\}$, we have the following corollary.

\begin{corollary}\label{Cor:ETimpliesAT}
Any edge-transitive propeller graph is arc-transitive as well.\end{corollary}

Some additional facts about these graphs are collected in the following lemmas.
\begin{lemma}\label{Lem:Equal}
\[\Pr\nolimits_{n}\left(b,c,d\right) = \Pr\nolimits_{n}\left(b,c,-d\right).\]
\end{lemma}
Notice that these two graphs are equal, and not merely isomorphic.

\begin{lemma}\label{Lem:Isom1}
\[\Pr\nolimits_{n}\left(b,c,d\right) \cong \Pr\nolimits_{n}\left(-b,c,d\right)\\
\Pr\nolimits_{n}\left(b,c,d\right) \cong \Pr\nolimits_{n}\left(b,-c,d\right)\]
\end{lemma}
\begin{proof}The isomorphisms are given below, respectively.
\begin{align*}
\left(A_{i} \mapsto A_{i-b}\right)\left(B_{i} \mapsto B_{i}\right)\left(C_{i} \mapsto C_{i}\right)\\
\left(A_{i} \mapsto A_{i}\right)\left(B_{i} \mapsto B_{i}\right)\left(C_{i} \mapsto C_{i-c}\right)
\end{align*}
\end{proof}

\begin{lemma}\label{Lem:Isom2}
Let $\left(d,n\right)=1$. Then there is an integer $0<e<n$ such that $de \EQ 1 \EQ ed$, and we have
\[\Pr\nolimits_{n}\left(b,c,d\right) \cong \Pr\nolimits_{n}\left(ce,be,e\right).\]
\end{lemma}
\begin{proof}
The isomorphism $\left(A_{i} \mapsto C_{ie}\right)\left(B_{i} \mapsto B_{ie}\right)\left(C_{i} \mapsto A_{ie}\right)$ delivers the result.

\end{proof}

Notice that in the special case where $d \in \{-1,1\}$, we have $\Pr_{n}\left(b,c,\pm 1\right) \cong \Pr_{n}\left(c,b,1\right)$.

We close the section with a couple of other observations about propeller graphs which will help us with the classification.

\begin{lem}A propeller graph $\Gamma$ admits an automorphism sending $\left(A_{0},A_{1}\right)$ to $\left(A_{0},B_{0}\right)$ if and only if $\Gamma$ is edge-transitive.\end{lem}
\label{Lem:WingtoFlat}



\begin{proof}
Let $\Gamma$ be a propeller graph with $G = Aut\left(\Gamma\right)$. Suppose there is an automorphism, say $\delta \in G$, which sends $\left(A_{0},A_{1}\right)$ to $\left(A_{0},B_{0}\right)$. By application of $\rho$ and $\mu$, this implies that $A$-wings and $A$-spokes are in the same orbit. Now, $\delta$ must send the remaining neighbors of $A_{1}$ to the remaining neighbors of $B_{0}$, that is, $\{A_{2},B_{1},B_{1-b}\}\delta = \{A_{b},C_{0},C_{c}\}$, so either $B_{1}\delta \in \{C_{0},C_{c}\}$ or $B_{1-b}\delta \in \{C_{0},C_{c}\}$. Hence some $A$-spoke, either $\{A_{1},B_{1}\}$ or $\{A_{1},B_{1-b}\}$, is sent to a $C$-spoke, either $\{B_{0},C_{0}\}$ or $\{B_{0},C_{c}\}$. This establishes that the $C$-spokes share an orbit with the $A$-wings and $A$-spokes. At this point, the $C$-wings are either a part of this orbit under $\delta$, or else $\delta$ preserves the $C$-wings. But $\delta$ cannot preserve the $C$-wings without naturally preserving the set of $C$-vertices, and since either $B_{1}\delta$ or $B_{1-b}\delta$ is a $C$-vertex, this is clearly not the case. Thus all edges are in a single orbit, which establishes the result.

\end{proof}

\begin{lem}A propeller graph has girth at most $6$.\end{lem}
\label{Lem:Girth}
\begin{proof}
Let $\Pr_{n}\left(b,c,d\right)$ be a propeller graph. If $b\EQ \pm 1$, then the graph admits either $\left(A_{0},A_{1},B_{1}\right)$ or $\left(A_{0},A_{1},B_{0}\right)$ as a $3$-cycle. Otherwise, the $6$-cycle $\left(A_{0},A_{1},B_{1},A_{1+b},A_{b},B_{0}\right)$ establishes the result. We use the terms $A$\emph{-canonical} and $C$\emph{-canonical} respectively to describe $6$-cycles of the forms below, and the term \emph{canonical }$6$\emph{-cycles} as a collective reference.
\[\left(A_{i},A_{i+1},B_{i+1},A_{i+1+b},A_{i+b},B_{i}\right)\\
\left(C_{i},C_{i+d},B_{i+d},C_{i+c+d},C_{i+c},B_{i}\right)\]

\end{proof}

\section{Families of Edge-Transitive Propeller Graphs}
\label{S:Families}

In this section, we introduce three infinite families and one finite family of edge-transitive propeller graphs. For each family, we will present an associated automorphism which we call the \emph{defining automorphism}, because a propeller graph admits such an automorphism if and only if it is a member of the corresponding family. Each defining automorphism also sends $\left(A_{0},A_{1}\right)$ to $\left(A_{0},B_{0}\right)$, thereby ensuring its family's edge-transitivity by Lemma \ref{Lem:WingtoFlat}.

\subsection{Family $1$}
The first family of propeller graphs we consider are those of the form $\Pr_{2m}\left(2d,2,d\right)$ with $d^2 \EQ 1$. All members of family $1$ admit the defining automorphism $\sigma_{1}$:

\[
\begin{tabular}{rr|ll}
 & \multicolumn{1}{c}{} & $i\overset{2}{\equiv}0$ & $i\overset{2}{\equiv}1$\tabularnewline
\cline{3-4} 
 & $A_{i}\mapsto$ & $A_{id}$ & $B_{(i-1)d}$\tabularnewline
$\sigma_{1}:$ & $B_{i}\mapsto$ & $A_{id+1}$ & $C_{(i-1)d+2}$\tabularnewline
 & $C_{i}\mapsto$ & $B_{id+1-2d}$ & $C_{(i-1)d+2-d}$\tabularnewline
\end{tabular}
\]

There are some special propeller graphs from family $1$ which we would like to point out. Firstly, the graph $\Pr_{4}\left(2,2,1\right)$ is the only propeller graph isomorphic to a wreath graph, namely, to $W(6,2)$. Wreath graphs admit many isolated, local automorphisms, which creates vertex-stabilizers exponential in the number of vertices \cite{TetraATGraphsUnboundedVS}. The interested reader may see \cite{FamRegGraphsRegMaps} for more information concerning wreath graphs.

Second, a special subset of family $1$ graphs admit an additional involution. These graphs are those for which $n = 6m$ for some $m$, $d\overset{6}{\equiv}5$, and $6d\overset{6m}{\equiv}6$. Since $6d\overset{6m}{\equiv}6$, either $3d\overset{6m}{\equiv}3$, or $m$ is even, $d\overset{12}{\equiv}5$, and $3d\overset{6m}{\equiv}3m+3$. We may denote these graphs as members of family $1^{\star}$. Now, let us fix an integer $r$ in the following way:
\[r = \begin{cases} 1 & \text{if }3d\overset{6m}{\equiv}3 \\ k+1 & \text{if } 3d\overset{6m}{\equiv}3m+3 \text{ and } m=2k.\end{cases}\]

It follows that $12r\overset{6m}{\equiv}12$, and if $3d\overset{6m}{\equiv}3$, $6r\overset{6m}{\equiv}6$, whereas if $3d\overset{6m}{\equiv}3m+3$ and $m=2k$, $6r\overset{6m}{\equiv}3m+6$. The following involution $\sigma_{1^\star}$ is admitted by this special subset of family $1$ graphs.

\[
\resizebox{\textwidth}{!}{
\begin{tabular}{rr|llllll}
& \multicolumn{1}{c}{} & $i\overset{6}{\equiv}0$ & $i\overset{6}{\equiv}1$ & $i\overset{6}{\equiv}2$ & $i\overset{6}{\equiv}3$ & $i\overset{6}{\equiv}4$ & $i\overset{6}{\equiv}5$\tabularnewline
\cline{3-8}
& $A_{i}\mapsto$ & $A_{ir}$ & $A_{(i-1)r+1}$ & $B_{(i-2)r+1}$ & $C_{(i-3)r+3}$ & $C_{(i-4)r+3+d}$ & $B_{(i-5)r+3+d}$\tabularnewline
$\sigma_{1^\star}:$ & $B_{i}\mapsto$ & $B_{ir}$ & $A_{(i-1)r+2}$ & $A_{(i-2)r+1+2d}$ & $B_{(i-3)r+3}$ & $C_{(i-4)r+3+2d}$ & $C_{(i-5)r+5+d}$\tabularnewline
& $C_{i}\mapsto$ & $C_{ir}$ & $B_{(i-1)r+2-2d}$ & $A_{(i-2)r+2d}$ & $A_{(i-3)r+3}$ & $B_{(i-4)r+1+2d}$ & $C_{(i-5)r+5}$\tabularnewline
\end{tabular}
}
\]

It is necessary to point out that, in the case where $m$ is even and $3d\overset{6m}{\equiv}3m+3$, the restriction that $d\overset{12}{\equiv}5$ is required. The propeller graph $\Pr_{6m}\left(2d,2,d\right)$ with $d^{2}\overset{6m}{\equiv}1$, $3d\overset{6m}{\equiv}3m+3$, and $d\overset{12}{\equiv}11$ does not admit the permutation $\sigma_{1^\star}$ as an automorphism, since $\sigma_{1^\star}$ sends the arc $\left(C_{0},C_{d}\right)$ to $\left(C_{0},C_{3m+d}\right)$, which is not an arc.

Of course, graphs in family $1^{\star}$ still admit $\sigma_{1}$, and in fact, $\sigma_{1^\star}\mu\sigma_{1^\star}\mu = \sigma_{1}$ and $\sigma_{1^\star}\rho\sigma_{1^\star}\rho^{-1} = \sigma_{1}$. Hence, in a sense, $\sigma_{1^\star}$ supercedes $\sigma_{1}$ as the defining automorphism for family $1^{\star}$.

\subsection{Family $2$}
The next family of propeller graphs are those of the form $\Pr_{2m}\left(2d,2,d\right)$ with $d^2\EQ -1$. We remark that although the parameters of these graphs are nearly identical to those of family $1$, their automorphism groups differ significantly enough to warrant treating them as a separate family.

The defining automorphism of family $2$ is given below.
\[
\begin{tabular}{rr|ll}
 & \multicolumn{1}{c}{} & $i\overset{2}{\equiv}0$ & $i\overset{2}{\equiv}1$\tabularnewline
\cline{3-4} 
 & $A_{i}\mapsto$ & $A_{id}$ & $B_{(i-1)d}$\tabularnewline
$\sigma_{2}:$ & $B_{i}\mapsto$ & $A_{id-1}$ & $C_{(i-1)d}$\tabularnewline
 & $C_{i}\mapsto$ & $B_{id-1-2d}$ & $C_{(i-1)d-d}$\tabularnewline
\end{tabular}
\]

\subsection{Family $3$}
The third and final infinite family of edge-transitive propeller graphs consists of those of the form $\Pr_{4m}\left(b,b-4,2b-3\right)$ with $b\overset{4}{\equiv}3$ and $8b\EQ 16$. The defining automorphism of family $3$ is given below.
\[
\begin{tabular}{rr|llll}
 & \multicolumn{1}{c}{} & $i\overset{4}{\equiv}0$ & $i\overset{4}{\equiv}1$ & $i\overset{4}{\equiv}2$ & $i\overset{4}{\equiv}3$\tabularnewline
\cline{3-6} 
 & $A_{i}\mapsto$ & $A_{i}$ & $B_{i-1}$ & $C_{i-2}$ & $B_{i+1-b}$\tabularnewline
$\sigma_{3}:$ & $B_{i}\mapsto$ & $A_{i+1}$ & $A_{i-1+b}$ & $C_{i-5+2b}$ & $C_{i+1-b}$\tabularnewline
 & $C_{i}\mapsto$ & $A_{i+2}$ & $B_{i-1+b}$ & $C_{i-8+4b}$ & $B_{i+5-2b}$\tabularnewline
\end{tabular}
\]

\subsection{Family $4$}
This family differs significantly from the others in that it only has finitely many members, namely, $\Pr_{5}\left(1,2,2\right)$, $\Pr_{10}\left(1,2,2\right)$, $\Pr_{10}\left(1,7,7\right)$, $\Pr_{10}\left(6,2,7\right)$, and $\Pr_{10}\left(6,7,2\right)$. The graph $\Pr_{10}\left(6,2,7\right)=\Pr_{10}\left(6,2,3\right)$ is uniquely interesting in that it admits both the defining automorphism of family $2$ and the defining automorphism of family $4$, shown below; that is, it is a member of both families. This is the only example of an edge-transitive propeller graph in multiple families.
\[
\begin{tabular}{rr|lllll}
 & \multicolumn{1}{c}{} & $i\overset{5}{\equiv}0$ & $i\overset{5}{\equiv}1$ & $i\overset{5}{\equiv}2$ & $i\overset{5}{\equiv}3$ & $i\overset{5}{\equiv}4$\tabularnewline
\cline{3-7} 
 & $A_{i}\mapsto$ & $A_{i}$ & $B_{i-1}$ & $C_{i-2+c}$ & $C_{i+b}$ & $B_{i-1+b}$\tabularnewline
$\sigma_{4}:$ & $B_{i}\mapsto$ & $A_{i+1}$ & $C_{i-1}$ & $B_{i-2+c}$ & $C_{i+b+d}$ & $A_{i-1+b}$\tabularnewline
 & $C_{i}\mapsto$ & $B_{i+1}$ & $B_{i-1-c}$ & $A_{i-2+c}$ & $C_{i+b+2d}$ & $A_{i-2+b}$\tabularnewline
\end{tabular}
\]

The automorphism $\sigma_{4}$ requires that $10\EQ 0$, $b\overset{5}{\equiv}1$, $c\overset{5}{\equiv}d\overset{5}{\equiv}2$, $5+b+c+d\EQ 0$, $2c\EQ 2d\EQ -6$, and $2b\EQ 2$. In particular, the first implication forces the family's finiteness.

\begin{remark}
This family presents more than its fair share of open questions regarding propeller graphs. The investigation into arc-transitive propeller graphs began in part by examining the image of the $A$\emph{-wing cycle}
\[\left(A_{0},A_{1},A_{2},\ldots,A_{-1}\right)\]
of known arc-transitive propeller graphs under an automorphism sending $\left(A_{0},A_{1}\right)$ to $\left(A_{0},B_{0}\right)$. In many cases, this automorphism sent the $A$-wing cycle to a cycle $\left(A_{0},B_{0},A_{b},B_{b},A_{2b},\ldots,B_{-b}\right)$, the investigation of which yielded families $1$ and $2$. In others, this automorphism sent the $A$-wing cycle to a cycle $\left(A_{0},B_{0},C_{0},B_{-c},A_{b-c},B_{b-c},\ldots,B_{-b}\right)$, characteristic of the graphs in family 3. However, only the five graphs in this family appeared to have an automorphism sending the $A$-wing cycle to a cycle
\[\left(A_{0},B_{0},C_{c},C_{c+d},B_{c+d},A_{b+c+d},\ldots,B_{-b}\right).\]
It would be interesting to discover a deeper reason as to why this particular property should ultimately disqualify all but a handful of arc-transitive propeller graphs.
\end{remark}

\section{Classification of Edge-Transitive Propeller Graphs}
\label{S:Classification}

We now turn to the classification of edge-transitive propeller graphs. For the rest of this paper, we assume that $\Gamma = \Pr_{n}\left(b,c,d\right)$ is an edge-transitive propeller graph with $G = Aut\left(\Gamma\right)$ and girth $g$. By Corollary \ref{Cor:ETimpliesAT}, $\Gamma$ is arc-transitive also. As $\Gamma$ is arc-transitive, every arc is contained in the same number of $3$-cycles, the same number of $4$-cycles, and so on. We use the following definition for clarity.

\begin{defn}
Let $\Gamma$ be a propeller graph, and let $N_3$, $N_4$, $N_5$, and $N_6$ denote the number of $3$-, $4$-, $5$-, and $6$-cycles of $\Gamma$, respectively.
\end{defn}

Note that because of the existence of canonical $6$-cycles, $N_6 \ge 2$ for any propeller graph. For the classification proofs, it will help us to consider only ``sufficiently large'' graphs, so by using the computer algebra program SAGE \cite{SAGE}, we can see that the following lemma holds.
\begin{lemma}
\label{Lem:lessthan78}
Let $\Gamma = \Pr_{n}\left(b,c,d\right)$ be an edge-transitive propeller graph where $n\le 78$. Then, up to isomorphism, $\Gamma$ is a member of family $1$, family $2$, family $3$, or family $4$.
\end{lemma}

Over the next three sections, we show that edge-transitive propeller graphs must be members of the families from Section \ref{S:Families}. Lemma \ref{Lem:Girth} implies that all propeller graphs have girth at most $6$, so we tackle the classification in cases: The $g=3$ case will be handled in Section \ref{S:g=3}, the $g=4$ case will be handled in Section \ref{S:g=4}, and the $g\ge 5$ case will be handled in Section \ref{S:g>4}.

\section{Where $g = 3$}
\label{S:g=3}

\begin{thm}Let $\Gamma$ be an edge-transitive propeller graph of girth $3$. Then $\Gamma$ is isomorphic to the line graph of $GP\left(n,d\right)$, a $2$-arc-transitive generalized Petersen graph.
\end{thm}
\begin{proof}
Since $\Gamma$ has girth $3$, by looking at possible $3$-cycles containing $\left(B_{0},A_{b}\right)$ and their induced relations, we must have that $b\EQ \pm 1$, and by considering the $3$-cycles containing $\left(B_{0},C_{c}\right)$, we also determine that $c\EQ \pm d$. Moreover, we see that $N_3 = 1$. From Lemma \ref{Lem:Isom1} we may assume without loss of generality that $b\EQ -1$ and $c\EQ -d$.

Now consider the generalized Petersen graph $X = GP\left(n,d\right)$. According to the construction given in \cite{GroupsGenPet}, its vertex set has $2n$ vertices, labeled $u_{i}$ and $v_{i}$ for $i\in\mathbb{Z}_{n}$, and its edge set consists of $3$ types of edges, with $n$ of each type: $\{u_{i},u_{i+1}\}$, $\{u_{i},v_{i}\}$, and $\{v_{i},v_{i+d}\}$, for $i\in\mathbb{Z}_{n}$. The line graph of $X$, written $\mathbb{L}\left(X\right)$, has vertex set $E(X)$, and two edges of $X$ are adjacent in $\mathbb{L}\left(X\right)$ if and only if they share a vertex in $X$. If we write the vertices of $\mathbb{L}\left(X\right)$ as $A_{i}$, $B_{i}$, and $C_{i}$ (for $i\in\mathbb{Z}_n$), corresponding respectively to $\{u_{i},u_{i+1}\}$, $\{u_{i},v_{i}\}$, and $\{v_{i},v_{i+d}\}$, then the edge set of $\mathbb{L}\left(X\right)$ has six types of edges: $\{A_{i},A_{i+1}\}$, $\{A_{i},B_{i}\}$, $\{B_{i},A_{i-1}\}$, $\{B_{i},C_{i-d}\}$, $\{C_{i},B_{i}\}$, and $\{C_{i},C_{i+d}\}$.

This makes it apparent that $\mathbb{L}\left(X\right) = \Gamma$. The arcs in $\Gamma$ correspond directly to the $2$-arcs in $X$, and as $\Gamma$ is arc-transitive, $X$ must be $2$-arc-transitive.

\begin{table}[h]
\centering
\begin{tabular}{|r|l|}
\hline
\textbf{$GP\left(n,r\right)$} & \textbf{$\Pr_{n}\left(b,c,d\right)$}\\
\hline
$GP\left(4,1\right)$ & $\Pr_{4}\left(1,1,1\right)$\\
$GP\left(5,2\right)$ & $\Pr_{5}\left(1,2,2\right)$\\
$GP\left(8,3\right)$ & $\Pr_{8}\left(1,3,3\right)$\\
$GP\left(10,2\right)$ & $\Pr_{10}\left(1,2,2\right)$\\
$GP\left(10,3\right)$ & $\Pr_{10}\left(1,3,3\right)$\\
$GP\left(12,5\right)$ & $\Pr_{12}\left(1,5,5\right)$\\
$GP\left(24,5\right)$ & $\Pr_{24}\left(1,5,5\right)$\\
\hline
\end{tabular}
\caption{The arc-transitive gen. Petersen graphs and their corresponding propeller graphs.}
\label{Tab:Girth3}
\end{table}

\end{proof}

Of course, the $2$-arc-transitive generalized Petersen graphs are well known to be the seven classified in \cite{GroupsGenPet}, which in particular have $n\le 24$. Therefore by Lemma \ref{Lem:lessthan78}, arc-transitive propeller graphs of girth $3$ are members of the families from Section \ref{S:Families}. In Table \ref{Tab:Girth3} we see the arc-transitive generalized Petersen graphs and their corresponding propeller graphs. We remark that the propeller graphs in Table \ref{Tab:Girth3} for which $n$ is a multiple of $4$ are isomorphic to members of family $3$, while the others are isomorphic to members of family $4$.

\section{Where $g = 4$}
\label{S:g=4}

In this section, we classify the arc-transitive propeller graphs of girth $4$.

\begin{theorem}Let $\Gamma$ be an arc-transitive propeller graph of girth $4$. Then $\Gamma$ is isomorphic to a graph of the form $\Pr_{2m}\left(2,2,d\right)$ with $d\in\{1,m+1\}$.\end{theorem}
\begin{proof}
A simple analysis of all possible $4$-cycles containing $\left(A_{0},A_{1}\right)$ shows that $\left(A_{0},A_{1}\right)$ can be in at most five $4$-cycles. Moreover, if $\left(A_{0},A_{1}\right)$ is in at least three $4$-cycles, it must be in all five; hence $N_{4}\in\{1,2,5\}$. If $N_{4} = 1$, or $N_{4} = 5$, then $n = 4$, and a simple SAGE search \cite{SAGE} shows that the only arc-transitive propeller graph with $n=4$ and $g=4$ is $\Pr_{4}\left(2,2,1\right)$.

However, if $N_{4} = 2$, then $b\EQ \pm 2$, $c\EQ \pm 2$, and $2d\EQ \pm 2$. By our isomorphism theorems, we may take $b\EQ 2$, $c\EQ 2$, and $2d\EQ 2$, in which case either $d\EQ 1$, or $n = 2m$ and $d\EQ m+1$. If $n = 2m$ and $d\EQ m+1$, we are done. Therefore, we proceed assuming that $d\EQ 1$, and show that $n$ must be even. By arc-transitivity, we know that there is an automorphism $\sigma$ such that $\left(A_{0},A_{1}\right)\sigma = \left(A_{0},B_{0}\right)$. The two $4$-cycles $\left(A_{0},A_{1},A_{2},B_{0}\right)$ and $\left(A_{0},A_{1},B_{-1},A_{-1}\right)$ are sent to the $4$-cycles $\left(A_{0},B_{0},A_{2},A_{1}\right)$ and $\left(A_{0},B_{0},C_{0},B_{-2}\right)$ in some fashion. Therefore $\{A_{2},B_{-1}\}\sigma = \{A_{2},C_{0}\}$ and $\{B_{0},A_{-1}\}\sigma = \{A_{1},B_{-2}\}$.

Observe that for a given $i$,
\begin{align*}
N(B_{i}) \cap N(A_{i+1}) &= \{A_{i},A_{i+2}\}, \\
N(B_{i}) \cap N(C_{i+1}) &= \{C_{i},C_{i+2}\}, \text{ and} \\
N(A_{i+2}) \cap N(C_{i+2}) &= \{B_{i},B_{i+2}\}.
\end{align*}
This implies that the full action of $\sigma$ is entirely determined by where it sends $A_{0}$, $A_{1}$, $B_{0}$, $C_{0}$, and $C_{1}$.

If $A_{2}\sigma = C_{0}$, then $B_{0}\sigma = B_{-2}$ and in fact \[\left(A_{0},A_{1},B_{1},A_{3},A_{2},B_{0}\right)\sigma = \left(A_{0},B_{0},C_{2},C_{1},C_{0},B_{-2}\right).\] We also have that $B_{-1}\sigma = A_{2}$ and $A_{-1}\sigma = A_{1}$, and therefore \[\left(A_{0},A_{1},B_{-1},A_{-1},A_{-2},B_{-2}\right)\sigma = \left(A_{0},B_{0},A_{2},A_{1},B_{-1},A_{-1}\right).\] This allows us to deduce that \[\left(B_{-2},A_{0},A_{1},B_{1},C_{1},C_{0}\right)\sigma = \left(A_{-1},A_{0},B_{0},C_{2},C_{1},B_{-1}\right).\] However, the arc $\left(B_{0},C_{0}\right)$ is clearly not preserved by $\sigma$, so $A_{2}\sigma \ne C_{0}$. Hence $A_{2}\sigma = A_{2}$, which yields a whole host of deductions:
\begin{align*}
	\left(A_{0},A_{1},B_{1},A_{3},A_{2},B_{0}\right)\sigma = \left(A_{0},B_{0},C_{2},B_{2},A_{2},A_{1}\right) \\
	\left(A_{0},A_{1},B_{-1},A_{-1},A_{-2},B_{-2}\right)\sigma = \left(A_{0},B_{0},C_{0},B_{-2},A_{-2},A_{-1}\right) \\
	\left(B_{-2},A_{0},A_{1},B_{1},C_{1},C_{0}\right)\sigma = \left(A_{-1},A_{0},B_{0},C_{2},C_{1},B_{-1}\right)
\end{align*}

This delivers the condition necessary to determine the full action of $\sigma$, shown below.
\[
\begin{tabular}{r|ll}
\multicolumn{1}{c}{} & $i\overset{2}{\equiv}0$ & $i\overset{2}{\equiv}1$\tabularnewline
\cline{2-3}
$A_{i}\mapsto$ & $A_{i}$ & $B_{i-1}$\tabularnewline
$B_{i}\mapsto$ & $A_{i+1}$ & $C_{i+1}$\tabularnewline
$C_{i}\mapsto$ & $B_{i-1}$ & $C_{i}$\tabularnewline
\end{tabular}
\]

It is immediate that this requires $n = 2m$ for some $m$. (In fact, this is the defining automorphism for family $1$ with $d\EQ 1$.)
\end{proof}

The classification of edge-transitive propeller graphs with girth $4$ follows immediately; namely, they are members of family $1$.

\section{Where $g \ge 5$}
\label{S:g>4}

We begin with a definition which streamlines some later proofs considerably.

\begin{defn}Let $\Gamma$ be an edge-transitive propeller graph of girth at least $5$, with $\left(u,v\right)$ an arc of $\Gamma$. Label the three remaining non-$u$ neighbors of $v$ as $w,x,y$, and the three remaining non-$v$ neighbors of $u$ as $r,s,t$. We define the \emph{successor type} of $\left(u,v\right)$ to be the multiset $\{a_{1},a_{2},a_{3}\}$, where $a_{1}$, $a_{2}$, and $a_{3}$ are the numbers of $6$-cycles containing $\left(u,v,w\right)$, $\left(u,v,x\right)$, and $\left(u,v,y\right)$, respectively. We similarly define the \emph{predecessor type} of $\left(u,v\right)$ to be the multiset $\{b_{1},b_{2},b_{3}\}$, where $b_{1}$, $b_{2}$, and $b_{3}$ are the numbers of $6$-cycles containing $\left(r,u,v\right)$, $\left(s,u,v\right)$, and $\left(t,u,v\right)$, respectively. As an aside, there are many occasions where $i\ne j$ but $a_{i} = a_{j}$ or $b_{i} = b_{j}$ for some $i,j$, and it is for this reason that we define types as multisets rather than merely sets.\end{defn}

A couple of things are readily apparent. First,
\[a_{1}+a_{2}+a_{3} = N_{6} = b_{1}+b_{2}+b_{3}.\]
Also, automorphisms of the graph preserve types, so for an edge-transitive propeller graph, every arc must have the same successor type and the same predecessor type. The next lemma establishes a lower bound of $N_{6} \ge 3$ for edge-transitive propeller graphs, and provides an example illustrating the previous definition.
\begin{lemma}There exists no arc-transitive propeller graph for which $N_{6} = 2$.\end{lemma}
\begin{proof}
Let $\Pr_{n}\left(b,c,d\right)$ be edge-transitive with $N_{6} = 2$; then every arc is in only canonical $6$-cycles. One of the canonical $6$-cycles containing $\left(A_{0},A_{1}\right)$ contains the $2$-arc $\left(A_{0},A_{1},B_{1}\right)$ and the other contains the $2$-arc $\left(A_{0},A_{1},B_{1-b}\right)$, so the successor type of $\left(A_{0},A_{1}\right)$ is $\{1,1\}$. Conversely, both canonical $6$-cycles containing $\left(A_{0},B_{0}\right)$ contain the $2$-arc $\left(A_{0},B_{0},A_{b}\right)$, so the successor type of $\left(A_{0},B_{0}\right)$ is $\{2,0\}$. Clearly there can be no automorphism sending $\left(A_{0},A_{1}\right)$ to $\left(A_{0},B_{0}\right)$, contradicting our assumption that the graph is arc-transitive.

\end{proof}

It will help to know precisely which, and how many, $6$-cycles are admitted by a particular propeller graph, as every non-canonical $6$-cycle admitted by $\Gamma$ imposes additional conditions on $n$, $b$, $c$, and $d$. For $\mathcal{X}$ an arbitrary cycle in $\Gamma$, let $q(\mathcal{X})$, $r(\mathcal{X})$, $s(\mathcal{X})$, $t(\mathcal{X})$, $u(\mathcal{X})$, and $v(\mathcal{X})$ be, respectively, the number of $A$-wings, $A$-flats, $A$-blades, $C$-blades, $C$-flats, and $C$-wings contained in $\mathcal{X}$. As an example, if $\mathcal{X}=\left(A_{0},A_{1},B_{1},A_{1+b},A_{b},B_{0}\right)$, we have $q(\mathcal{X})=r(\mathcal{X})=s(\mathcal{X})=2$ and $t(\mathcal{X})=u(\mathcal{X})=v(\mathcal{X})=0$.

We say that two $6$-cycles are \emph{of the same class} if there is an automorphism from $\left<\rho,\mu\right>\le G$ sending one onto the other. Note that if $\mathcal{X}$ and $\mathcal{Y}$ are of the same class, then $q(\mathcal{X})=q(\mathcal{Y})$, $r(\mathcal{X})=s(\mathcal{Y})$, $s(\mathcal{X})=r(\mathcal{Y})$, $t(\mathcal{X})=u(\mathcal{Y})$, $u(\mathcal{X})=t(\mathcal{Y})$, and $v(\mathcal{X})=v(\mathcal{Y})$. In Table \ref{Tab:6CycRels}, found in the Appendix, we list a representative $6$-cycle $\mathcal{X}$ of each class, conditions on $n$, $b$, $c$, and $d$ for $\mathcal{X}$ to be admitted by $\Gamma$, and values for $q(\mathcal{X})$, $r(\mathcal{X})$, $s(\mathcal{X})$, $t(\mathcal{X})$, $u(\mathcal{X})$, and $v(\mathcal{X})$ when $\mathcal{X}$ is of the given class.

\begin{remark}We are seeking which combinations of relations from Table \ref{Tab:6CycRels} yield potentially edge-transitive graphs. There are of course a terrific number of possible combinations to start with, but we are able to reduce the possibilities considerably in three ways. First, by applying key assumptions from our claim, namely, that $g\ge 5$ and $n>78$, some relations are thereby rendered mutually exclusive. For example, taking any two relations from
\[\{b\EQ 4, b\EQ -4, 2b\EQ 2, 2b\EQ -2\}\]
will contradict $n>78$, so numbers $4$, $5$, $6$, and $7$ from Table \ref{Tab:6CycRels} are all pairwise mutually exclusive.

Second, we must have also that all edges are in the same number of $6$-cycles, which requires that the sum of the values of $q(E)$, $r(E)$, $s(E)$, $t(E)$, $u(E)$, and $v(E)$ are all the same across all representative $6$-cycles $E$. For example, if $2b\EQ 2$ and $2c\EQ 2d$, then every edge type is in five $6$-cycles (two canonical $6$-cycles and three associated with one of the relations). For this reason, a combination like $b\EQ 2c$ and $d\EQ 1$ would not be viable on its own, since edges of different types would be contained in a different number of $6$-cycles. By applying the first and second restrictions, we reduce the number of feasible combinations down to $222$ possibilities.

Finally, we utilize Lemmas \ref{Lem:Equal}, \ref{Lem:Isom1}, and \ref{Lem:Isom2} from Section \ref{S:PropGraphs} to reduce the $222$ possibilities to $31$ non-isomorphic cases, which can be found in Table \ref{Tab:31cases}. Notice also that all of these cases adhere to the upper bound $N_{6}\le 9$.
\end{remark}

\begin{table}
\[
\begin{tabular}{|l|c|l|}
\cline{1-3}
Case & $N_{6}$ & Propeller graph\tabularnewline
\cline{1-3}
$1$ & $3$ & $\Pr_{n}\left(b,c,1+b+c\right)$\tabularnewline
\cline{1-3}
$2$ & $4$ & $\Pr_{n}\left(b,c,1+b\right)$\tabularnewline
$3$ & $4$ & $\Pr_{n}\left(b,c,1+c\right)$\tabularnewline
$4$ & $4$ & $\Pr_{n}\left(b,c,1+b-c\right)$, where $2c\EQ 2$\tabularnewline
$5$ & $4$ & $\Pr_{n}\left(b,c,1+b-c\right)$ where $2b\EQ 2c$\tabularnewline
$6$ & $4$ & $\Pr_{n}\left(b,2+b,d\right)$ where $2+2b\EQ 2d$\tabularnewline
$7$ & $4$ & $\Pr_{n}\left(b,2+b,d\right)$ where $2d\EQ 2$\tabularnewline
\cline{1-3}
$8$ & $5$ & $\Pr_{n}\left(b,c,d\right)$ where $2b\EQ 2$, $2c\EQ 2d$\tabularnewline
$9$ & $5$ & $\Pr_{n}\left(b,b-2,3\right)$ where $2b\EQ 8$\tabularnewline
$10$ & $5$ & $\Pr_{n}\left(b,2+b,3\right)$ where $2b\EQ 4$\tabularnewline
$11$ & $5$ & $\Pr_{n}\left(b,b-2,2b-1\right)$ where $6b\EQ 4$\tabularnewline
$12$ & $5$ & $\Pr_{n}\left(b,b-2,2b-3\right)$ where $6b\EQ 8$\tabularnewline
$13$ & $5$ & $\Pr_{n}\left(b,2+b,d\right)$ where $4b\EQ 0$, $2d\EQ 2$\tabularnewline
$14$ & $5$ & $\Pr_{n}\left(b,b-2,d\right)$ where $4b\EQ 8$, $2d\EQ 2$\tabularnewline
\cline{1-3}
$15$ & $6$ & $\Pr_{n}\left(b,c,1\right)$\tabularnewline
$16$ & $6$ & $\Pr_{n}\left(2d,2,d\right)$\tabularnewline
$17$ & $6$ & $\Pr_{n}\left(b,c,1+b-c\right)$ where $2b\EQ 2$, $2c\EQ 2d$\tabularnewline
$18$ & $6$ & $\Pr_{n}\left(b,b-2,2b-1\right)$ where $4b\EQ 4$\tabularnewline
$19$ & $6$ & $\Pr_{n}\left(b,2+b,b-1\right)$ where $4b\EQ 0$\tabularnewline
$20$ & $6$ & $\Pr_{n}\left(b,2+b,b-1\right)$ where $2b\EQ 4$\tabularnewline
$21$ & $6$ & $\Pr_{n}\left(b,b-2,b-3\right)$ where $4b\EQ 8$\tabularnewline
$22$ & $6$ & $\Pr_{n}\left(b,b-2,b-3\right)$ where $2b\EQ 8$\tabularnewline
\cline{1-3}
$23$ & $7$ & $\Pr_{n}\left(b,c,1-c\right)$ where $2b\EQ 2$, $4c\EQ 2$\tabularnewline
$24$ & $7$ & $\Pr_{n}\left(b,c,1-b+c\right)$ where $2b\EQ 2$\tabularnewline
$25$ & $7$ & $\Pr_{n}\left(b,c,1+b\right)$ where $2b\EQ 2$, $2c\EQ 4$\tabularnewline
$26$ & $7$ & $\Pr_{n}\left(2d,2,d\right)$ where $3d\EQ 3$\tabularnewline
\cline{1-3}
$27$ & $8$ & $\Pr_{n}\left(b,2+b,1+b\right)$\tabularnewline
\cline{1-3}
$28$ & $9$ & $\Pr_{n}\left(b,c,1\right)$ where $2b\EQ 2$, $2c\EQ 2$\tabularnewline
$29$ & $9$ & $\Pr_{n}\left(b,2+b,1\right)$\tabularnewline
$30$ & $9$ & $\Pr_{n}\left(2d,2,d\right)$ where $3d\EQ 1$\tabularnewline
$31$ & $9$ & $\Pr_{n}\left(6,2,3\right)$\tabularnewline
\cline{1-3}
\end{tabular}
\]
\caption{Possible edge-transitive propeller graphs of girth at least $5$ and $n>78$.}
\label{Tab:31cases}
\end{table}

This brings us to our next theorem, which classifies the edge-transitive propeller graphs of girth at least $5$ and $n>78$.

\begin{thm}Let $\Gamma = \Pr_{n}\left(b,c,d\right)$ be one of the graphs shown in the table above. If $\Gamma$ is edge-transitive, exactly one of the following must be true:
\begin{enumerate}[(i)]
\item $\Gamma$ is a graph from case $1$ or a graph from case $5$, and is isomorphic to a member of family $3$.
\item $\Gamma$ is a graph from case $16$, and is isomorphic to either a member of family $1$ or a member of family $2$.
\item $\Gamma$ is a graph from case $26$, and is isomorphic to a member of family $1$.
\end{enumerate}
\end{thm}

First, we will prove that edge-transitive propeller graphs from cases $1$, $5$, $16$, or $26$ are isomorphic to members of the known families from Section \ref{S:Families}. We will then show that propeller graphs from any of the other cases cannot be edge-transitive.

\begin{proof}
(Case $1$)
Let $\Pr_{n}\left(b,c,1+b+c\right)$ be edge-transitive. We observe that $N_{6} = 3$, so every edge is in two canonical $6$-cycles as well as a cycle of the form $\left(A_{i},A_{i+1},B_{i+1},C_{i+1+c},C_{i-b},B_{i-b}\right)$ for some $i$. We claim that the full action of an automorphism is determined by its action on $A_{i-1}$, $A_{i}$, $B_{i}$, and $A_{i+b}$. Observe that the two $6$-cycles containing $\left(A_{i},B_{i},A_{i+b}\right)$ are
\[\left(A_{i},B_{i},A_{i+b},A_{i-1+b},B_{i-1},A_{i-1}\right) \text{ and } \left(A_{i},B_{i},A_{i+b},A_{i+1+b},B_{i+1},A_{i+1}\right),\]
and these are distinguishable by whether they contain $A_{i-1}$ or not. Hence knowing the images of these four vertices under the automorphism will allow us to determine where it sends the second $6$-cycle, and specifically where it sends $A_{i+1}$, $B_{i+1}$, and $A_{i+1+b}$. This process can be repeated finitely many times to deduce the images of all $A$ and $B$ vertices in the graph. We may then look to the images of the non-canonical $6$-cycles to determine the images of the $C$ vertices also.

By edge-transitivity and Corollary \ref{Cor:ETimpliesAT}, there exists an automorphism $\alpha$ such that $\left(A_{0},A_{1}\right)\alpha = \left(A_{0},B_{0}\right)$, and so $\alpha$ must send the three $6$-cycles containing $\left(A_{0},A_{1}\right)$ to the three $6$-cycles containing $\left(A_{0},B_{0}\right)$ in some way. There are two $6$-cycles containing $\left(B_{-b},A_{0},A_{1}\right)$ and two $6$-cycles containing $\left(A_{-1},A_{0},B_{0}\right)$, so $B_{-b}\alpha = A_{-1}$. Also, there are two $6$-cycles containing each of $\left(A_{0},A_{1},B_{1}\right)$ and $\left(A_{0},B_{0},A_{b}\right)$, so $B_{1}\alpha = A_{b}$. This gives us enough information to determine the action of $\alpha$ on the $6$-cycles containing $\left(A_{0},A_{1}\right)$, and in particular, we have
\[\left(A_{0},A_{1},B_{1},A_{1+b},A_{b},B_{0}\right)\alpha = \left(A_{0},B_{0},A_{b},A_{1+b},B_{1},A_{1}\right).\]
Taking $i=1$, this allows us to determine the full action of $\alpha$ per the previous paragraph. This is what we obtain for $\alpha$ (where $z$ is an integer such that $4z\EQ b-c$):

\[
\begin{tabular}{rr|llll}
 & \multicolumn{1}{c}{} & $i\overset{4}{\equiv}0$ & $i\overset{4}{\equiv}1$ & $i\overset{4}{\equiv}2$ & $i\overset{4}{\equiv}3$\tabularnewline
\cline{3-6} 
 & $A_{i}\mapsto$ & $A_{iz}$ & $B_{(i-1)z}$ & $C_{(i-2)z}$ & $B_{(i-3)z-c}$\tabularnewline
$\alpha:$ & $B_{i}\mapsto$ & $A_{iz+1}$ & $A_{(i-1)z+b}$ & $C_{(i-2)z+1+b+c}$ & $C_{(i-3)z-c}$\tabularnewline
 & $C_{i}\mapsto$ & $A_{iz+2}$ & $B_{(i-1)z+b}$ & $C_{(i-2)z+2+2b+2c}$ & $B_{(i-3)z-2c}$\tabularnewline
\end{tabular}
\]

Here $b\overset{4}{\equiv}c\overset{4}{\equiv}3$, and we know already that $\left(A_{b},B_{0}\right)\alpha = \left(B_{1},A_{1}\right)$. Since $C_{c}\alpha = B_{(c-3)z-2c}$, we have that $(c-3)z-2c\EQ 1-b$, and therefore that $(c-3)z\EQ 1-b+2c\EQ 1+c-4z$. Equivalently, $(c+1)z\EQ 1+c$. But when we see that $\left(C_{1+c},B_{1},A_{1+b}\right)\alpha = \left(A_{3+c},A_{b},A_{1+b}\right)$, we find that $b\EQ 4+c$, and hence $4z\EQ 4$. All instances of $z$ in our formula for $\alpha$ above can thus be removed, revealing that $\alpha$ is precisely the defining automorphism of family $3$. We conclude that all graphs isomorphic to $\Pr_{n}\left(b,c,1+b+c\right)$ with $N_{6} = 3$ are isomorphic to members of family $3$.

\end{proof}

\begin{proof}
(Case $5$)
Let $\Pr_{n}\left(b,c,1+b-c\right)$ where $2b\EQ 2c$ and $N_{6} = 4$ be edge-transitive. From edge-transitivity and Corollary \ref{Cor:ETimpliesAT}, there is an automorphism $\beta$ such that $\left(A_{0},A_{1}\right)\beta = \left(A_{0},B_{0}\right)$, and so must send the four $6$-cycles containing $\left(A_{0},A_{1}\right)$ to those containing $\left(A_{0},B_{0}\right)$ in some way.

As with case $1$, if for a particular $i$, we know $\left(A_{i-1},A_{i},B_{i},A_{i+b}\right)\beta$, we can also find $\left(A_{i+1},B_{i+1},A_{i+1+b}\right)\beta$, and so on, thereby determining how $\beta$ acts on all $A$ and $B$ vertices. To obtain $C_{i}\beta$, we need $\left(B_{i-1},A_{i-1},A_{i},B_{i}\right)\beta$, because one of the two $6$-cycles containing $\left(A_{i},A_{i+1},B_{i+1}\right)$ contains $C_{i+1}$ and the other contains $B_{i}$. Therefore, all $C$ vertices are determined as well.

Now, we have two choices for how $\beta$ acts on $B_{0}$ and $B_{1}$, and these choices will determine the action of $\beta$ on the $6$-cycles containing $\left(A_{0},A_{1}\right)$, thereby determining the full action of $\beta$. Specifically we have that $B_{0}\beta\in\{A_{1},A_{-1}\}$ and $B_{1}\beta\in\{A_{b},C_{0}\}$. Let us denote by $\beta_{1}$ and $\beta_{2}$ the versions of $\beta$ where $B_{0}\beta = A_{1}$ and, respectively, $B_{1}\beta = A_{b}$ or $B_{1}\beta = C_{0}$. Let us also use $\beta_{3}$ and $\beta_{4}$ for the versions of $\beta$ where $B_{0}\beta = A_{-1}$ and, respectively, $B_{1}\beta = A_{b}$ or $B_{1}\beta = C_{0}$.

Let $z$ be an integer such that $4z\EQ b+c$. Then $\beta_{1}$ and $\beta_{2}$ are the following:
\[
\begin{tabular}{rr|llll}
 & \multicolumn{1}{c}{} & $i\overset{4}{\equiv}0$ & $i\overset{4}{\equiv}1$ & $i\overset{4}{\equiv}2$ & $i\overset{4}{\equiv}3$\tabularnewline
\cline{3-6} 
 & $A_{i}\mapsto$ & $A_{iz}$ & $B_{(i-1)z}$ & $C_{(i-2)z+c}$ & $B_{(i-3)z+c}$\tabularnewline
$\beta_{1}:$ & $B_{i}\mapsto$ & $A_{iz+1}$ & $A_{(i-1)z+b}$ & $C_{(i-2)z+1+b}$ & $C_{(i-3)z+2c}$\tabularnewline
 & $C_{i}\mapsto$ & $B_{iz+1-b}$ & $A_{(i-1)z-1+b}$ & $B_{(i-2)z+1+b-c}$ & $C_{(i-3)z-1+b+c}$\tabularnewline
\end{tabular}
\]

\[
\begin{tabular}{rr|llll}
 & \multicolumn{1}{c}{} & $i\overset{4}{\equiv}0$ & $i\overset{4}{\equiv}1$ & $i\overset{4}{\equiv}2$ & $i\overset{4}{\equiv}3$\tabularnewline
\cline{3-6} 
 & $A_{i}\mapsto$ & $A_{iz}$ & $B_{(i-1)z}$ & $C_{(i-2)z+c}$ & $B_{(i-3)z+c}$\tabularnewline
$\beta_{2}:$ & $B_{i}\mapsto$ & $A_{iz+1}$ & $C_{(i-1)z}$ & $C_{(i-2)z+1+b}$ & $A_{(i-3)z+c}$\tabularnewline
 & $C_{i}\mapsto$ & $B_{iz+1}$ & $C_{(i-1)z-1-b+c}$ & $B_{(i-2)z+1+b}$ & $A_{(i-3)z-1+c}$\tabularnewline
\end{tabular}
\]

The automorphism $\beta_{1}$ requires that $n=4m$ for some $m$, that $b\overset{4}{\equiv}3$, and that $c\overset{4}{\equiv}1$. Also, $\left(A_{0},B_{0},C_{c}\right)\beta_{1} = \left(A_{0},A_{1},A_{(c-1)z-1+b}\right)$, so $\left(c-1\right)z\EQ 3-b$, and $\left(A_{1},B_{1},C_{1+c}\right)\beta_{1} = \left(B_{0},A_{b},B_{(c-1)z+1+b-c}\right)$, so $\left(c-1\right)z\EQ c-1$. This implies that $b+c\EQ 4$, which suggests that $4z\EQ 4$. Also, since $2b\EQ 2c$, we have $4b\EQ 8$. A graph of case $5$ admitting $\beta_{1}$ as an automorphism must therefore be isomorphic to a member of family $3$.

Conversely, the automorphism $\beta_{2}$ requires that $n=4m$ for some $m$, that $b\overset{4}{\equiv}1$, and that $c\EQ 3$. As well, $\left(A_{0},B_{0},C_{c}\right)\beta_{2} = \left(A_{0},A_{1},A_{(c-3)z-1+c}\right)$, so $\left(c-3\right)z\EQ 3-c$. Continuing, $\left(A_{1},B_{1},C_{1+c}\right)\beta_{2} = \left(B_{0},C_{0},B_{(1+c)z+1}\right)$, so $\left(1+c\right)z\EQ -1-c$. Therefore $\left(1+c\right)z\EQ 3+b$, and so $4+b+c\EQ 0$. Also, $d\EQ -3-2b$ and because $2b\EQ 2c$, we have $4b\EQ -8$. Clearly this graph is then isomorphic to a member of family $3$.

Next, we come to $\beta_{3}$ and $\beta_{4}$.
\[
\begin{tabular}{rr|llll}
 & \multicolumn{1}{c}{} & $i\overset{4}{\equiv}0$ & $i\overset{4}{\equiv}1$ & $i\overset{4}{\equiv}2$ & $i\overset{4}{\equiv}3$\tabularnewline
\cline{3-6} 
 & $A_{i}\mapsto$ & $A_{iz}$ & $B_{(i-1)z}$ & $C_{(i-2)z+c}$ & $B_{(i-3)z+c}$\tabularnewline
$\beta_{3}:$ & $B_{i}\mapsto$ & $A_{iz-1}$ & $A_{(i-1)z+b}$ & $C_{(i-2)z-1+b}$ & $C_{(i-3)z+2c}$\tabularnewline
 & $C_{i}\mapsto$ & $B_{iz-1-b}$ & $A_{(i-1)z+1+b}$ & $B_{(i-2)z-1+b-c}$ & $C_{(i-3)z+1+b+c}$\tabularnewline
\end{tabular}
\]

\[
\begin{tabular}{rr|llll}
 & \multicolumn{1}{c}{} & $i\overset{4}{\equiv}0$ & $i\overset{4}{\equiv}1$ & $i\overset{4}{\equiv}2$ & $i\overset{4}{\equiv}3$\tabularnewline
\cline{3-6} 
 & $A_{i}\mapsto$ & $A_{iz}$ & $B_{(i-1)z}$ & $C_{(i-2)z+c}$ & $B_{(i-3)z+c}$\tabularnewline
$\beta_{4}:$ & $B_{i}\mapsto$ & $A_{iz-1}$ & $C_{(i-1)z}$ & $C_{(i-2)z-1+b}$ & $A_{(i-3)z+c}$\tabularnewline
 & $C_{i}\mapsto$ & $B_{iz-1}$ & $C_{(i-1)z+1-b+c}$ & $B_{(i-2)z-1+b}$ & $A_{(i-3)z+1+c}$\tabularnewline
\end{tabular}
\]

Both $\beta_{3}$ and $\beta_{4}$ require that $n = 4m$ for some $m$. It is immediate that $\beta_{3}$ requires $b\overset{4}{\equiv}3$ and $c\overset{4}{\equiv}1$ and $\beta_{4}$ requires $b\overset{4}{\equiv}1$ and $c\overset{4}{\equiv}3$. However, we observe that there are vertices in the graph which $\beta_{3}$ and $\beta_{4}$ can never map to. For example, no vertex will every be sent to $A_{1}$ by $\beta_{3}$, because, for any $k$, $A_{4k}$ and $C_{4k+1}$ are sent to $A$ vertices with index $0$ mod $4$, and $B_{4k}$ and $B_{4k+1}$ are sent to $A$ vertices with index $3$ mod $4$. The same problem arises with $\beta_{4}$ as well. Therefore neither $\beta_{3}$ nor $\beta_{4}$ can be automorphisms.

\end{proof}

\begin{proof}
(Case $16$)
Let $\Pr_{n}\left(2d,2,d\right)$ be edge-transitive with $N_{6} = 6$. We list the six $6$-cycles containing each edge type below.

\[
\resizebox{\textwidth}{!}{
\begin{tabular}{|l|ll|}
\cline{1-3}
Arc & $6$-cycles & \tabularnewline
\cline{1-3}
 & $\left(A_{0},A_{1},A_{2},B_{2},C_{2},B_{0}\right)$ & $\left(A_{0},A_{1},A_{2},B_{2-2d},C_{2-2d},B_{-2d}\right)$\tabularnewline
$\left(A_{0},A_{1}\right)$ & $\left(A_{0},A_{1},B_{1},A_{1+2d},A_{2d},B_{0}\right)$ & $\left(A_{0},A_{1},B_{1},C_{1},B_{-1},A_{-1}\right)$\tabularnewline
 & $\left(A_{0},A_{1},B_{1-2d},A_{1-2d},A_{-2d},B_{-2d}\right)$ & $\left(A_{0},A_{1},B_{1-2d},C_{1-2d},B_{-1-2d},A_{-1}\right)$\tabularnewline
\cline{1-3}
 & $\left(A_{0},B_{0},A_{2d},A_{1+2d},B_{1},A_{1}\right)$ & $\left(A_{0},B_{0},A_{2d},A_{-1+2d},B_{-1},A_{-1}\right)$\tabularnewline
$\left(A_{0},B_{0}\right)$ & $\left(A_{0},B_{0},C_{0},B_{-2},A_{-2},A_{-1}\right)$ & $\left(A_{0},B_{0},C_{0},C_{-d},C_{-2d},B_{-2d}\right)$\tabularnewline
 & $\left(A_{0},B_{0},C_{2},B_{2},A_{2},A_{1}\right)$ & $\left(A_{0},B_{0},C_{2},C_{2-d},C_{2-2d},B_{-2d}\right)$\tabularnewline
\cline{1-3}
 & $\left(B_{0},A_{2d},A_{1+2d},A_{2+2d},B_{2},C_{2}\right)$ & $\left(B_{0},A_{2d},A_{1+2d},B_{1},A_{1},A_{0}\right)$\tabularnewline
$\left(B_{0},A_{2d}\right)$ & $\left(B_{0},A_{2d},A_{-1+2d},A_{-2+2d},B_{-2},C_{0}\right)$ & $\left(B_{0},A_{2d},A_{-2d},B_{-1},A_{-1},A_{0}\right)$\tabularnewline
 & $\left(B_{0},A_{2d},B_{2d},C_{2d},C_{d},C_{0}\right)$ & $\left(B_{0},A_{2d},B_{2d},C_{2+2d},C_{2+d},C_{2}\right)$\tabularnewline
\cline{1-3}
 & $\left(B_{0},C_{2},B_{2},A_{2},A_{1},A_{0}\right)$ & $\left(B_{0},C_{2},B_{2},A_{2+2d},A_{1+2d},A_{2d}\right)$\tabularnewline
$\left(B_{0},C_{2}\right)$ & $\left(B_{0},C_{2},C_{2+d},B_{d},C_{d},C_{0}\right)$ & $\left(B_{0},C_{2},C_{2+d},C_{2+2d},B_{2d},A_{2d}\right)$\tabularnewline
 & $\left(B_{0},C_{2},C_{2-d},B_{-d},C_{-d},C_{0}\right)$ & $\left(B_{0},C_{2},C_{2-d},C_{2-2d},B_{-2d},A_{0}\right)$\tabularnewline
\cline{1-3}
 & $\left(C_{0},B_{0},A_{0},A_{-1},A_{-2},B_{-2}\right)$ & $\left(C_{0},B_{0},A_{0},B_{-2d},C_{-2d},C_{-d}\right)$\tabularnewline
$\left(C_{0},B_{0}\right)$ & $\left(C_{0},B_{0},A_{2d},A_{-1+2d},A_{-2+2d},B_{-2}\right)$ & $\left(C_{0},B_{0},A_{2d},B_{2d},C_{2d},C_{d}\right)$\tabularnewline
 & $\left(C_{0},B_{0},C_{2},C_{2+d},B_{d},C_{d}\right)$ & $\left(C_{0},B_{0},C_{2},C_{2-d},B_{-d},C_{-d}\right)$\tabularnewline
\cline{1-3}
 & $\left(C_{0},C_{d},B_{d},A_{d},B_{-d},C_{-d}\right)$ & $\left(C_{0},C_{d},B_{d},C_{2+d},C_{2},B_{0}\right)$\tabularnewline
$\left(C_{0},C_{d}\right)$ & $\left(C_{0},C_{d},B_{-2+d},A_{-2+d},B_{-2-d},C_{-d}\right)$ & $\left(C_{0},C_{d},B_{-2+d},C_{-2+d},C_{-2},B_{-2}\right)$\tabularnewline
 & $\left(C_{0},C_{d},C_{2d},B_{2d},A_{2d},B_{0}\right)$ & $\left(C_{0},C_{d},C_{2d},B_{-2+2d},A_{-2+2d},B_{-2}\right)$\tabularnewline
\cline{1-3}
\end{tabular}
}
\]

By edge-transitivity and Corollary \ref{Cor:ETimpliesAT}, there is an automorphism $\gamma$ such that $\left(A_{0},A_{1}\right)\gamma = \left(A_{0},B_{0}\right)$. We notice that every $6$-cycle above is uniquely determined by identifying a $3$-arc contained in it; for example, although there are two $6$-cycles containing $\left(A_{0},A_{1},B_{1}\right)$, only one contains the $3$-arc $\left(B_{0},A_{0},A_{1},B_{1}\right)$, and only one contains the $3$-arc $\left(A_{-1},A_{0},A_{1},B_{1}\right)$. Moreover, none contain the $3$-arc $\left(B_{-2d},A_{0},A_{1},B_{1}\right)$.

Now, we claim that we can deduce the entire action of $\gamma$ on the vertices of the graph from its action on $A_{i}$, $A_{i+1}$, $B_{i}$, and $B_{i+1}$, for some fixed integer $i$. To establish this, we shall prove that if we know the images of these four vertices under $\gamma$, then we can obtain the images of $A_{i+2}$, $B_{i+2}$, and $C_{i+2}$.

The two $6$-cycles containing $\left(B_{i},A_{i},A_{i+1}\right)$ are
\[\left(B_{i},A_{i},A_{i+1},A_{i+2},B_{i+2},C_{i+2}\right)\]
and
\[\left(B_{i},A_{i},A_{i+1},B_{i+1},A_{i+1+2d},A_{i+2d}\right).\]
If we know how $\gamma$ acts on $A_{i}$, $A_{i+1}$, $B_{i}$, and $B_{i+1}$, then we can immediately determine the image of the latter $6$-cycle, since it is the only one which contains the $3$-arc $\left(B_{0}\gamma,A_{0}\gamma,A_{1}\gamma,B_{1}\gamma\right)$. From there, the image of the former $6$-cycle is immediately known also, giving us the images of $A_{i+2}$, $B_{i+2}$, and $C_{i+2}$.

So now consider that we already have $A_{0}\gamma = A_{0}$ and $A_{1}\gamma = B_{0}$ by assumption. From the list of $6$-cycles above, we see that $B_{0}\gamma \in \{A_{1},A_{-1},B_{-2d}\}$. If $B_{0}\gamma = A_{1}$, then $B_{1}\gamma \in \{A_{2d},C_{2}\}$, but if $B_{0}\gamma = A_{-1}$, then $B_{1}\gamma \in \{A_{2d},C_{0}\}$, whereas if $B_{0}\gamma = B_{-2d}$, then $B_{1}\gamma \in \{C_{0},C_{2}\}$. This gives us six possible actions of $\gamma$ on the graph, which are denoted below by $\gamma_{i}$ for $i\in \{1,2,3,4,5,6\}$.

\[
\begin{tabular}{c|c}
$\gamma_{i}$ & Assumptions\tabularnewline
\cline{1-2}
$\gamma_{1}$: & $B_{0}\mapsto A_{1}$ and $B_{1}\mapsto A_{2d}$\tabularnewline
$\gamma_{2}$: & $B_{0}\mapsto A_{1}$ and $B_{1}\mapsto C_{2}$\tabularnewline
$\gamma_{3}$: & $B_{0}\mapsto A_{-1}$ and $B_{1}\mapsto A_{2d}$\tabularnewline
$\gamma_{4}$: & $B_{0}\mapsto A_{-1}$ and $B_{1}\mapsto C_{0}$\tabularnewline
$\gamma_{5}$: & $B_{0}\mapsto B_{-2d}$ and $B_{1}\mapsto C_{0}$\tabularnewline
$\gamma_{6}$: & $B_{0}\mapsto B_{-2d}$ and $B_{1}\mapsto C_{2}$\tabularnewline
\end{tabular}
\]

The automorphism $\gamma_{1}$ is the following (where $z$ is an integer such that $6z\EQ 3+3d$):
\[
\resizebox{\textwidth}{!}{
\begin{tabular}{rr|llllll}
& \multicolumn{1}{c}{} & $i\overset{6}{\equiv}0$ & $i\overset{6}{\equiv}1$ & $i\overset{6}{\equiv}2$ & $i\overset{6}{\equiv}3$ & $i\overset{6}{\equiv}4$ & $i\overset{6}{\equiv}5$\tabularnewline
\cline{3-8}
& $A_{i}\mapsto$ & $A_{iz}$ & $B_{(i-1)z}$ & $C_{(i-2)z+2}$ & $C_{(i-3)z+2+d}$ & $B_{(i-4)z+2+d}$ & $A_{(i-5)z+2+3d}$\tabularnewline
$\gamma_{1}:$ & $B_{i}\mapsto$ & $A_{iz+1}$ & $A_{(i-1)z+2d}$ & $B_{(i-2)z+2}$ & $C_{(i-3)z+2+2d}$ & $C_{(i-4)z+4+d}$ & $B_{(i-5)z+2+3d}$\tabularnewline
& $C_{i}\mapsto$ & $B_{iz+1-2d}$ & $A_{(i-1)z-1+2d}$ & $A_{(i-2)z+2}$ & $B_{(i-3)z+2d}$ & $C_{(i-4)z+4}$ & $C_{(i-5)z+2+3d}$\tabularnewline
\end{tabular}
}
\]

Immediately we see that $n = 6m$ for some $m$. Also, clearly $b$ (that is, $2d$) cannot be $1$, $2$, $3$, or $5$ (mod $6$), and $d$ cannot be $0$, $2$, $3$, or $4$ (mod $6$). For example, if $b\overset{6}{\equiv}2$, then we have $\left(B_{0},A_{2d}\right)\gamma_{1} = \left(A_{1},C_{(2d-2)z+2}\right)$, which of course is not an arc of the graph. We therefore have that $d\overset{6}{\equiv}5$. Notice also that $\left(C_{0},B_{0},A_{2d}\right)\gamma_{1} = \left(B_{1-2d},A_{1},B_{(2d-4)z+2+d}\right)$, so that $\left(2d-4\right)z\EQ -1-d$. Notice therefore that $\left(2d+2\right)z\EQ 2+2d$. Furthermore, $\left(B_{0},C_{2},C_{2+d}\right)\gamma_{1} = \left(A_{1},A_{2},A_{(1+d)z-1+2d}\right)$, so $\left(1+d\right)z\EQ 4-2d$. Therefore $\left(2d+2\right)z\EQ 8-4d\EQ 2+2d$, implying that $6d\EQ 6$.

If we set $d = 6g-1$ for some $g$, then $6d = 36g-6\EQ 6$, implying that $36g\EQ 12$. Then $d^2\EQ 36g^2-12g+1\EQ 12g-12g+1\EQ 1$, and the graph is therefore a member of family $1$. More specifically, it must be a member of family $1^{\star}$.

The automorphism $\gamma_{2}$ is the defining automorphism of family $1$.

The automorphism $\gamma_{3}$ is the following (where $z$ is an integer such that $6z\EQ -3+3d$):
\[
\resizebox{\textwidth}{!}{
\begin{tabular}{rr|llllll}
& \multicolumn{1}{c}{} & $i\overset{6}{\equiv}0$ & $i\overset{6}{\equiv}1$ & $i\overset{6}{\equiv}2$ & $i\overset{6}{\equiv}3$ & $i\overset{6}{\equiv}4$ & $i\overset{6}{\equiv}5$\tabularnewline
\cline{3-8}
& $A_{i}\mapsto$ & $A_{iz}$ & $B_{(i-1)z}$ & $C_{(i-2)z}$ & $C_{(i-3)z+d}$ & $B_{(i-4)z-2+d}$ & $A_{(i-5)z-2+3d}$\tabularnewline
$\gamma_{3}:$ & $B_{i}\mapsto$ & $A_{iz-1}$ & $A_{(i-1)z+2d}$ & $B_{(i-2)z-2}$ & $C_{(i-3)z+2d}$ & $C_{(i-4)z-2+d}$ & $B_{(i-5)z-2+3d}$\tabularnewline
& $C_{i}\mapsto$ & $B_{iz-1-2d}$ & $A_{(i-1)z+1+2d}$ & $A_{(i-2)z-2}$ & $B_{(i-3)z+2d}$ & $C_{(i-4)z-2}$ & $C_{(i-5)z+3d}$\tabularnewline
\end{tabular}
}
\]

It is immediate that $n = 6m$ for some $m$, and that $d\overset{6}{\equiv}5$.

However, notice that no vertex is sent via $\gamma_{3}$ to $A_{1}$; hence $\gamma_{3}$ cannot actually be an automorphism.

The automorphism $\gamma_{4}$ is the defining automorphism for family $2$.

The automorphism $\gamma_{5}$ must be the following (where $z$ is an integer such that $6z\EQ 3+3d$):
\[
\resizebox{\textwidth}{!}{
\begin{tabular}{rr|llllll}
& \multicolumn{1}{c}{} & $i\overset{6}{\equiv}0$ & $i\overset{6}{\equiv}1$ & $i\overset{6}{\equiv}2$ & $i\overset{6}{\equiv}3$ & $i\overset{6}{\equiv}4$ & $i\overset{6}{\equiv}5$\tabularnewline
\cline{3-8}
& $A_{i}\mapsto$ & $A_{iz}$ & $B_{(i-1)z}$ & $C_{(i-2)z+2}$ & $C_{(i-3)z+2+d}$ & $B_{(i-4)z+2+d}$ & $A_{(i-5)z+2+3d}$\tabularnewline
$\gamma_{5}:$ & $B_{i}\mapsto$ & $B_{iz-2d}$ & $C_{(i-1)z}$ & $C_{(i-2)z+2-d}$ & $B_{(i-3)z+d}$ & $A_{(i-4)z+2+d}$ & $A_{(i-5)z+1+3d}$\tabularnewline
& $C_{i}\mapsto$ & $A_{iz-2d}$ & $B_{(i-1)z-2}$ & $C_{(i-2)z+2-2d}$ & $C_{(i-3)z+d}$ & $B_{(i-4)z+2-d}$ & $A_{(i-5)z+3d}$\tabularnewline
\end{tabular}
}
\]
It is immediate that $d\overset{6}{\equiv}1$. However, we notice that there is no vertex which may be sent by $\gamma_{6}$ to $A_{1}$, so $\gamma_{5}$ must not be an automorphism after all.

The automorphism $\gamma_{6}$ is the following (where $z$ is an integer such that $6z\EQ -3+3d$):
\[
\resizebox{\textwidth}{!}{
\begin{tabular}{rr|llllll}
& \multicolumn{1}{c}{} & $i\overset{6}{\equiv}0$ & $i\overset{6}{\equiv}1$ & $i\overset{6}{\equiv}2$ & $i\overset{6}{\equiv}3$ & $i\overset{6}{\equiv}4$ & $i\overset{6}{\equiv}5$\tabularnewline
\cline{3-8}
& $A_{i}\mapsto$ & $A_{iz}$ & $B_{(i-1)z}$ & $C_{(i-2)z}$ & $C_{(i-3)z+d}$ & $B_{(i-4)z-2+d}$ & $A_{(i-5)z-2+3d}$\tabularnewline
$\gamma_{6}:$ & $B_{i}\mapsto$ & $B_{iz-2d}$ & $C_{(i-1)z+2}$ & $C_{(i-2)z-d}$ & $B_{(i-3)z+d}$ & $A_{(i-4)z-2+d}$ & $A_{(i-5)z-1+3d}$\tabularnewline
& $C_{i}\mapsto$ & $A_{iz-2d}$ & $B_{(i-1)z+2}$ & $C_{(i-2)z-2d}$ & $C_{(i-3)z+2+d}$ & $B_{(i-4)z-2-d}$ & $A_{(i-5)z+3d}$\tabularnewline
\end{tabular}
}
\]

Clearly $n=6m$ for some $m$, and it is simple to deduce that $d\overset{6}{\equiv} 1$. We will fish out additional conditions by calculating the images of some $2$-arcs under $\gamma_{6}$. First, $\left(B_{0},C_{0},C_{d}\right)\gamma_{6} = \left(B_{-2d},A_{-2d},B_{(d-1)z+2}\right)$, so $\left(d-1\right)z\EQ -2-4d$. Second, $\left(B_{1},C_{1},C_{1+d}\right)\gamma_{6} = \left(C_{2},B_{2},C_{(d-1)z-2d}\right)$, so $\left(d-1\right)z\EQ 4+2d$, implying that $6d\EQ -6$.

If we write $d=6g+1$ for some $g$, then $6d = 36g+6\EQ -6$, so that $36g\EQ -12$. Therefore $d^2 = 36g^2+12g+1\EQ -12g+12g+1\EQ 1$, which ensures that the graph is a member of family $1$. In fact, since $6d\EQ -6$, the graph is isomorphic to a member of family $1^{\star}$.

Hence graphs in case $16$ are isomorphic to members of family $1$, $2$, or $3$.

\end{proof}

\begin{proof}
(Case $26$)
Let $\Pr_{n}\left(2d,2,d\right)$ where $3d\EQ 3$ be edge-transitive. There must then exist an automorphism $\delta$ such that $\left(A_{0},A_{1}\right)\delta = \left(A_{0},B_{0}\right)$. We note that each arc in this graph has successor and predecessor type $\{2,2,3\}$, which allows us to determine that $B_{1-2d}\delta = C_{0}$, $B_{0}\delta = A_{1}$, $\{A_{2},B_{1}\}\delta = \{A_{2d},C_{2}\}$, and $\{A_{-1},B_{-2d}\}\delta = \{A_{-1},B_{-2d}\}$.

If $A_{2}\delta = A_{2d}$, we can deduce that $\delta$ is the defining automorphism of family $1$. If instead $A_{2}\delta = C_{2}$, we can prove that $\delta$ is the following:
\[
\begin{tabular}{r|llllll}
\multicolumn{1}{c}{} & $i\overset{6}{\equiv}0$ & $i\overset{6}{\equiv}1$ & $i\overset{6}{\equiv}2$ & $i\overset{6}{\equiv}3$ & $i\overset{6}{\equiv}4$ & $i\overset{6}{\equiv}5$\tabularnewline
\cline{2-7}
$A_{i}\mapsto$ & $A_{i}$ & $B_{i-1}$ & $C_{i}$ & $C_{i-1+d}$ & $B_{i-2+d}$ & $A_{i}$\tabularnewline
$B_{i}\mapsto$ & $A_{i+1}$ & $A_{i-1+2d}$ & $B_{i}$ & $C_{i-1+2d}$ & $C_{i+d}$ & $B_{i}$\tabularnewline
$C_{i}\mapsto$ & $B_{i+1-2d}$ & $A_{i-2+2d}$ & $A_{i}$ & $B_{i-3+2d}$ & $C_{i}$ & $C_{i}$\tabularnewline
\end{tabular}
\]
We observe that this is the defining automorphism of family $1^{\star}$. Hence edge-transitive graphs in case $26$ are, up to isomorphism, members of family $1$.

\end{proof}

We have now shown that edge-transitive propeller graphs from cases $1$, $5$, $16$, and $26$ must be members of the known families, up to isomorphism. We will now show that propeller graphs from any other case cannot be edge-transitive. Before doing that, we want to identify a particularly vexing case. In the course of this research, it was particularly difficult to find a proof that propeller graphs from case $24$ could not be arc-transitive. A strategy that worked in all other cases - namely, identifying the local or full action of an automorphism sending $\left(A_{0},A_{1}\right)$ to $\left(A_{0},B_{0}\right)$ - was not possible here, and a different method of attack was necessary. We have singled out the proof for case $24$ for this reason.

\begin{proof}
(Case $24$ contains no edge-transitive graphs)

Let $\Pr_{n}\left(b,c,1-b+c\right)$ where $2b\EQ 2$ and $2c\EQ 2d$ be edge-transitive with girth at least $5$ and $N_{6}=7$. We must have $n = 2m$ for some $m$ and $b\overset{2m}{\equiv}m+1$, because otherwise $b\EQ 1$ and the girth would be $3$. By the same reasoning, $c\overset{2m}{\equiv} m+d$, so we may write the graph as $\Gamma = \Pr_{2m}\left(m+1,m+d,d\right)$.

We note immediately that for any index $i$, $A_{i+m}$ is the unique vertex which is antipodal to $A_{i}$ in six distinct $6$-cycles. This property similarly holds for $B_{i}$ and $B_{i+m}$, and for $C_{i}$ and $C_{i+m}$. We illustrate this property in Figure \ref{Fig:Case24}, which shows the subgraphs induced by these $6$-cycles. Hence the sets $\{A_{i},A_{i+m}\}$, $\{B_{i},B_{i+m}\}$, $\{C_{i},C_{i+m}\}$ across $i\in \mathbb{Z}_m$ form a block system for the graph, and therefore $\left<\rho^m\right>$ is normal in the automorphism group of the graph.

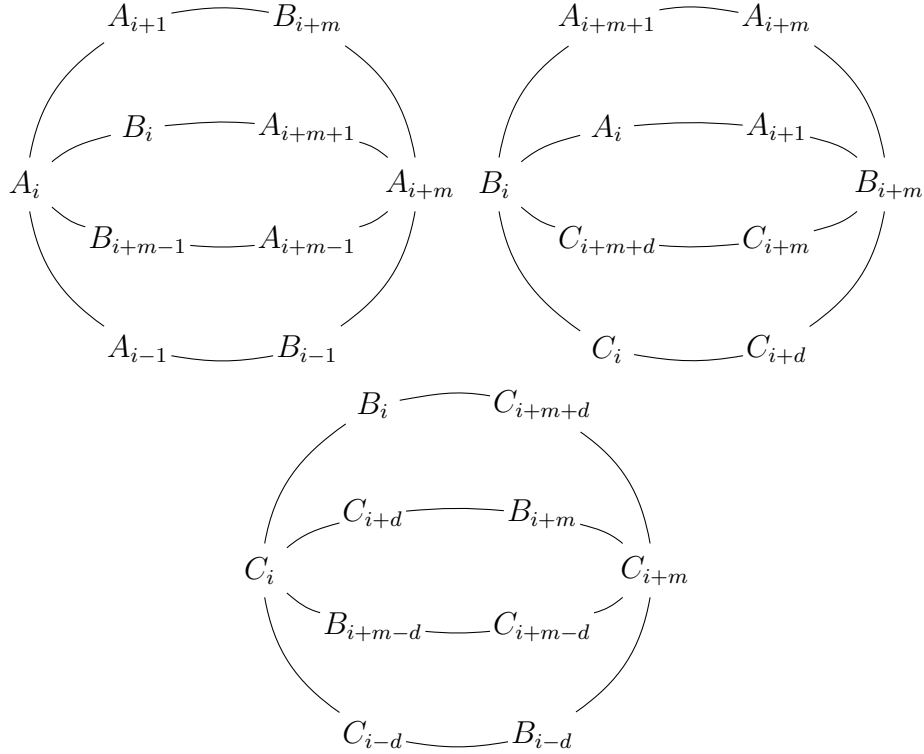
\begin{figure}[!ht]
\centering
\begin{tikzpicture}
	\newcommand \vertlarge {2.2}
    \newcommand \vertsmall {0.75}
    \newcommand \horizend {1.5}
    \newcommand \horizmid {2.25}
	\newcommand \angleone {80}
	\newcommand \angletwo {215}
	\newcommand \anglethree {40}
	\newcommand \anglefour {195}
	\newcommand \anglefive {10}
	\newcommand \anglesix {5}
	\tikzstyle{vertex} = [minimum size=20pt,inner sep=1pt]
	\node[vertex] (01) at (0,0) {$A_{i}$};
	\node[vertex] (02) at (\horizend,\vertlarge) {$A_{i+1}$};
	\node[vertex] (03) at (\horizend,\vertsmall) {$B_{i}$};
	\node[vertex] (04) at (\horizend,-\vertsmall) {$B_{i+m-1}$};
	\node[vertex] (05) at (\horizend,-\vertlarge) {$A_{i-1}$};
	\node[vertex] (06) at (\horizend+\horizmid,\vertlarge) {$B_{i+m}$};
	\node[vertex] (07) at (\horizend+\horizmid,\vertsmall) {$A_{i+m+1}$};
	\node[vertex] (08) at (\horizend+\horizmid,-\vertsmall) {$A_{i+m-1}$};
	\node[vertex] (09) at (\horizend+\horizmid,-\vertlarge) {$B_{i-1}$};
	\node[vertex] (10) at (2*\horizend+\horizmid,0) {$A_{i+m}$};
    \draw (01) edge[out=\angleone,in=\angletwo] (02);
    \draw (02) edge[out=\anglefive,in=180-\anglefive] (06);
    \draw (06) edge[out=180-\angletwo,in=180-\angleone] (10);
    \draw (01) edge[out=360-\angleone,in=360-\angletwo] (05);
    \draw (05) edge[out=360-\anglefive,in=180+\anglefive] (09);
    \draw (09) edge[out=180+\angletwo,in=180+\angleone] (10);
    \draw (01) edge[out=\anglethree,in=\anglefour] (03);
    \draw (03) edge[out=\anglesix,in=180-\anglesix] (07);
    \draw (07) edge[out=180-\anglefour,in=180-\anglethree] (10);
    \draw (01) edge[out=360-\anglethree,in=360-\anglefour] (04);
    \draw (04) edge[out=360-\anglesix,in=180+\anglesix] (08);
    \draw (08) edge[out=180+\anglefour,in=180+\anglethree] (10);
\end{tikzpicture}
\begin{tikzpicture}
	\newcommand \vertlarge {2.2}
    \newcommand \vertsmall {0.75}
    \newcommand \horizend {1.5}
    \newcommand \horizmid {2.25}
	\newcommand \angleone {80}
	\newcommand \angletwo {215}
	\newcommand \anglethree {40}
	\newcommand \anglefour {195}
	\newcommand \anglefive {10}
	\newcommand \anglesix {5}
	\tikzstyle{vertex} = [minimum size=20pt,inner sep=1pt]
	\node[vertex] (01) at (0,0) {$B_{i}$};
	\node[vertex] (02) at (\horizend,\vertlarge) {$A_{i+m+1}$};
	\node[vertex] (03) at (\horizend,\vertsmall) {$A_{i}$};
	\node[vertex] (04) at (\horizend,-\vertsmall) {$C_{i+m+d}$};
	\node[vertex] (05) at (\horizend,-\vertlarge) {$C_{i}$};
	\node[vertex] (06) at (\horizend+\horizmid,\vertlarge) {$A_{i+m}$};
	\node[vertex] (07) at (\horizend+\horizmid,\vertsmall) {$A_{i+1}$};
	\node[vertex] (08) at (\horizend+\horizmid,-\vertsmall) {$C_{i+m}$};
	\node[vertex] (09) at (\horizend+\horizmid,-\vertlarge) {$C_{i+d}$};
	\node[vertex] (10) at (2*\horizend+\horizmid,0) {$B_{i+m}$};
    \draw (01) edge[out=\angleone,in=\angletwo] (02);
    \draw (02) edge[out=\anglefive,in=180-\anglefive] (06);
    \draw (06) edge[out=180-\angletwo,in=180-\angleone] (10);
    \draw (01) edge[out=360-\angleone,in=360-\angletwo] (05);
    \draw (05) edge[out=360-\anglefive,in=180+\anglefive] (09);
    \draw (09) edge[out=180+\angletwo,in=180+\angleone] (10);
    \draw (01) edge[out=\anglethree,in=\anglefour] (03);
    \draw (03) edge[out=\anglesix,in=180-\anglesix] (07);
    \draw (07) edge[out=180-\anglefour,in=180-\anglethree] (10);
    \draw (01) edge[out=360-\anglethree,in=360-\anglefour] (04);
    \draw (04) edge[out=360-\anglesix,in=180+\anglesix] (08);
    \draw (08) edge[out=180+\anglefour,in=180+\anglethree] (10);
\end{tikzpicture}

\begin{tikzpicture}
	\newcommand \vertlarge {2.2}
    \newcommand \vertsmall {0.75}
    \newcommand \horizend {1.5}
    \newcommand \horizmid {2.25}
	\newcommand \angleone {80}
	\newcommand \angletwo {215}
	\newcommand \anglethree {40}
	\newcommand \anglefour {195}
	\newcommand \anglefive {10}
	\newcommand \anglesix {5}
	\tikzstyle{vertex} = [minimum size=20pt,inner sep=1pt]
	\node[vertex] (01) at (0,0) {$C_{i}$};
	\node[vertex] (02) at (\horizend,\vertlarge) {$B_{i}$};
	\node[vertex] (03) at (\horizend,\vertsmall) {$C_{i+d}$};
	\node[vertex] (04) at (\horizend,-\vertsmall) {$B_{i+m-d}$};
	\node[vertex] (05) at (\horizend,-\vertlarge) {$C_{i-d}$};
	\node[vertex] (06) at (\horizend+\horizmid,\vertlarge) {$C_{i+m+d}$};
	\node[vertex] (07) at (\horizend+\horizmid,\vertsmall) {$B_{i+m}$};
	\node[vertex] (08) at (\horizend+\horizmid,-\vertsmall) {$C_{i+m-d}$};
	\node[vertex] (09) at (\horizend+\horizmid,-\vertlarge) {$B_{i-d}$};
	\node[vertex] (10) at (2*\horizend+\horizmid,0) {$C_{i+m}$};
    \draw (01) edge[out=\angleone,in=\angletwo] (02);
    \draw (02) edge[out=\anglefive,in=180-\anglefive] (06);
    \draw (06) edge[out=180-\angletwo,in=180-\angleone] (10);
    \draw (01) edge[out=360-\angleone,in=360-\angletwo] (05);
    \draw (05) edge[out=360-\anglefive,in=180+\anglefive] (09);
    \draw (09) edge[out=180+\angletwo,in=180+\angleone] (10);
    \draw (01) edge[out=\anglethree,in=\anglefour] (03);
    \draw (03) edge[out=\anglesix,in=180-\anglesix] (07);
    \draw (07) edge[out=180-\anglefour,in=180-\anglethree] (10);
    \draw (01) edge[out=360-\anglethree,in=360-\anglefour] (04);
    \draw (04) edge[out=360-\anglesix,in=180+\anglesix] (08);
    \draw (08) edge[out=180+\anglefour,in=180+\anglethree] (10);
\end{tikzpicture}
\caption{In case $24$, certain pairs of vertices are antipodal in six distinct $6$-cycles.}
\label{Fig:Case24}
\end{figure}

Since $\left<\rho^m\right>$ is normal in $Aut\left(\Gamma\right)$, we can consider its quotient graph $\Upsilon$, with vertex set $\{u_{i},v_{i},w_{i}\vert i\in \mathbb{Z}_{m}\}$, where $u_{i}$ corresponds to $\{A_{i},A_{i+m}\}$ in $\Gamma$, $v_{i}$ corresponds to $\{B_{i},B_{i+m}\}$ in $\Gamma$, and $w_{i}$ corresponds to $\{C_{i},C_{i+m}\}$ in $\Gamma$. The edge set of $\Upsilon$ induced by modding out $\left<\rho^m\right>$ is
\[
\{\{u_{i},u_{i+1}\},\{u_{i},v_{i}\},\{v_{i},u_{i+1}\},\{v_{i},w_{i+d}\},\{w_{i},v_{i}\},\{w_{i},w_{i+d}\}\vert i\in\mathbb{Z}_m\},
\]

and clearly $\Upsilon = \Pr_{m}\left(1,d,d\right)$, a propeller graph of girth $3$. Moreover, it is not difficult to see that if $\Gamma$ is arc-transitive, then so is $\Upsilon$.

We have already seen that the largest (with respect to order) arc-transitive propeller graph of girth $3$ is $\Pr_{24}\left(1,5,5\right)$, so $\Gamma$ must have $m\le 24$. Therefore any arc-transitive graphs from case $24$ must have $n\le 48$ and therefore are members of the known families by Lemma \ref{Lem:lessthan78}.
\end{proof}

\begin{proof}
(All other cases contain no edge-transitive graphs)
In cases $2$, $3$, $4$, $8$, $12$, $13$, $14$, $15$, $17$, $18$, $19$, $21$, $23$, $25$, $27$, $28$, and $29$, $\left(A_{0},A_{1}\right)$ and $\left(A_{0},B_{0}\right)$ have different successor types, so graphs from these cases cannot be edge-transitive.

Let $\Gamma$ be a graph from case $6$, where $N_{6} = 4$. Let $S$ be the set of four $6$-cycles containing $\left(A_{0},A_{1}\right)$ and let $T$ be the set of $6$-cycles containing $\left(A_{0},B_{0}\right)$. We note that $S$ can be partitioned into two pairs: two $6$-cycles which contain the $2$-arc $\left(A_{0},A_{1},B_{1}\right)$, and two which contain the $2$-arc $\left(B_{-b},A_{0},A_{1}\right)$. Though there are two $6$-cycles in $T$ which contain the $2$-arc $\left(A_{0},B_{0},A_{b}\right)$ and two which contain the $2$-arc $\left(A_{-1},A_{0},B_{0}\right)$, they do not partition $T$. Therefore $\Gamma$ cannot admit an automorphism sending $\left(A_{0},A_{1}\right)$ to $\left(A_{0},B_{0}\right)$, and therefore cannot be edge-transitive.

The proof for case $7$ is nearly identical to the proof for case $6$.

Let $\Gamma$ be a graph from case $9$, so that $N_{6} = 5$. Suppose that $\Gamma$ is edge-transitive, and let $\sigma$ be an automorphism such that $\left(A_{0},A_{1}\right)\sigma = \left(A_{0},B_{0}\right)$. We will show that this assumption leads to a contradiction. A short investigation shows that the predecessor type and successor type are the same for $\Gamma$, specifically $\{2,2,1\}$. In particular, $U = \left(A_{0},A_{1},B_{1-b},C_{-1},B_{-1},A_{-1}\right)$ is the only $6$-cycle containing $\left(A_{-1},A_{0},A_{1}\right)$ and $V = \left(A_{0},A_{1},A_{2},B_{2-b},C_{0},B_{0}\right)$ is the only $6$-cycle containing $\left(A_{0},A_{1},A_{2}\right)$. Also, $W = \left(A_{0},B_{0},C_{-2+b},C_{-5+b},C_{-b},B_{-b}\right)$ is both the only $6$-cycle containing $\left(B_{-b},A_{0},B_{0}\right)$ and the only $6$-cycle containing $\left(A_{0},B_{0},C_{-2+b}\right)$. Hence $\sigma$ would need to send $U$ to $W$ but would also need to send $V$ to $W$. Clearly $\sigma$ cannot do both, yielding a contradiction as claimed.

The proof for case $10$ is nearly identical to the proof for case $9$.

Let $\Gamma$ be a graph from case $11$, so that $N_{6} = 5$. Suppose that $\Gamma$ is edge-transitive, and let $\sigma$ be an automorphism sending $\left(A_{0},A_{1}\right)$ to $\left(A_{0},B_{0}\right)$. We will show that this assumption leads to a contradiction. First note that the predecessor and successor types for $\Gamma$ are both $\{2,2,1\}$. Also, the $6$-cycle $\left(A_{0},A_{1},A_{2},B_{2-b},C_{0},B_{0}\right)$ is both the only $6$-cycle containing $\left(A_{0},A_{1},A_{2}\right)$ and the only $6$-cycle containing $\left(A_{0},B_{0},C_{0}\right)$, so $\sigma$ must take this $6$-cycle and reverse it about $A_{0}$. Similarly, the only $6$-cycle containing $\left(A_{-1},A_{0},A_{1}\right)$ is $\left(A_{-1},A_{0},A_{1},B_{1-b},C_{-1},B_{-1}\right)$ and the only $6$-cycle containing $\left(B_{-b},A_{0},B_{0}\right)$ is $\left(B_{-b},A_{0},B_{0},C_{-2+b},C_{-3+3b},B_{-b}\right)$, so $\sigma$ must send the former to the latter. This in turn implies that the $3$-arc $P = \left(B_{1-b},A_{1},A_{2},B_{2-b}\right)$ is mapped via $\sigma$ to $Q = \left(C_{-2+b},B_{0},C_{0},B_{2-b}\right)$. However, $P$ is contained in a $6$-cycle, while $Q$ is not; therefore $\sigma$ cannot be an automorphism, giving us a contradiction.

Let $\Gamma$ be a graph from case $20$, where $N_{6} = 6$. The successor type for all arcs in this graph is $\{1,2,3\}$. It is simple to show that there are three $6$-cycles containing $\left(A_{0},A_{1},B_{1-b}\right)$ and that two of these contain $B_{0}$. It is also easy to find that each of the three $6$-cycles containing $\left(A_{0},B_{0},C_{0}\right)$ contain distinct neighbors of $A_{0}$. Therefore $\Gamma$ cannot admit an automorphism sending $\left(A_{0},A_{1}\right)$ to $\left(A_{0},B_{0}\right)$.

The proof for case $22$ is nearly identical to the proof for case $20$.

Let $\Gamma$ be a graph from case $30$, where $N_{6} = 9$. The predecessor type and successor type for all arcs in $\Gamma$ is $\{2,3,4\}$. Now, suppose that $\Gamma$ is edge-transitive, and let $\sigma$ be an automorphism sending $\left(A_{0},A_{1}\right)$ to $\left(A_{0},B_{0}\right)$. There are precisely four $6$-cycles containing $\left(A_{0},A_{1},B_{d}\right)$ and four containing $\left(A_{0},B_{0},C_{0}\right)$; therefore $B_{d}\sigma = C_{0}$. Similarly, there are four $6$-cycles containing each of $\left(B_{0},A_{0},A_{1}\right)$ and $\left(A_{1},A_{0},B_{0}\right)$, so $B_{0}\sigma = A_{1}$. However, while there are two $6$-cycles containing the $3$-arc $U = \left(B_{0},A_{0},A_{1},B_{d}\right)$, there is only one containing $V = \left(A_{1},A_{0},B_{0},C_{0}\right)$, and yet $U\sigma = V$. This gives the expected contradiction.

The proof for case $31$ is nearly identical to the proof for case $30$.
\end{proof}

We have established that all edge-transitive propeller graphs with girth at least $5$ and $n>78$ must be, up to isomorphism, a member of family $1$, family $2$, or family $3$. With Lemma \ref{Lem:lessthan78}, this completes the classification of edge-transitive propeller graphs of girth at least $5$.

\section{Additional avenues of study}
\label{S:Further}

We have been able to identify the automorphism groups of all arc-transitive propeller graphs; these results will be presented in a future paper. Meanwhile, the identification of automorphism groups of non-arc-transitive propeller graphs is still an open question worth pursuing.

Also, the classificaiton of vertex-transitive propeller graphs is still open, although preliminary investigations have yielded a few conjectures.

\begin{conjecture}
Let $\Gamma$ be a vertex-transitive propeller graph with $G=Aut\left(\Gamma\right)$. Then one of the following must be true:
\begin{itemize}
\item $\Gamma$ is arc-transitive, or
\item $\Gamma \cong \Pr_{2m}\left(2d,2,d\right)$ where $d$ is odd, in which case either $m\ne 2d$ and $\vert G\vert = 12m$, or $m=2d$ and $\vert G\vert = 48m$.
\end{itemize}
\end{conjecture}

\begin{conjecture}
Let $\Gamma$ be a propeller graph with two edge-orbits and $G=Aut(\left(\Gamma\right)$. Then one of the following must be true:
\begin{itemize}
\item $\Gamma$ is vertex-transitive, or
\item $\Gamma \cong \Pr_{2m}\left(m,m,d\right)$ where $\left(d,2m\right)=1$, in which case $\vert G\vert = 2^{m+3}\cdot 2m$, or
\item $\Gamma \cong \Pr_{n}\left(b,bd,d\right)$ where $d^2\EQ \pm 1$, in which case $\vert G\vert = 4n$.
\end{itemize}
\end{conjecture}

\begin{conjecture}
Let $\Gamma$ be a propeller graph with four edge-orbits and $G=Aut\left(\Gamma\right)$. Then one of the following must be true:
\begin{itemize}
\item $G = \left<\rho,\mu\right>$ is dihedral of order $2n$ (this is the default condition), or
\item $\Gamma \cong \Pr_{4d}\left(2d,2d,d\right)$, in which case $\vert G\vert = 2^{4d+1}\cdot 4d$, or
\item $\Gamma \cong \Pr_{2m}\left(m,m,d\right)$ where $2d<m$, in which case $\vert G\vert = 2^{m+(d,m)+1}\cdot 2m$, or
\item $\Gamma \cong \Pr_{4d}\left(b,2d,d\right)$, in which case $\vert G\vert = 2^{2d+1}\cdot 4d$, or
\item $\Gamma \cong \Pr_{2m}\left(b,m,d\right)$ where $2d<m$, in which case $\vert G\vert = 2^{1+(d,m)}\cdot 2m$, or
\item $\Gamma \cong \Pr_{2m}\left(m,c,d\right)$ where one of the following holds:
	\begin{itemize}[$\circ$]
		\item $2c=2d=m$, in which case $\vert G\vert = 2^{m+1}\cdot 2m$, or
		\item $2c=m$, $2d\ne m$, $c>1$, $\left(d,m\right)>1$, in which case $\vert G\vert = 2^{1+(d,m)}\cdot 2m$, or
		\item $2c\ne m$, $2d=m$, $\left(c,m\right)>1$, $d>1$, in which case $\vert G\vert = 2^{1+(c,m)}\cdot 2m$, or
		\item $\vert G\vert = 2^{1+(m,c,d)}\cdot 2m$.
	\end{itemize}
\end{itemize}
\end{conjecture}

There is also the question of generalizing the family of propeller graphs. This is a highly precedented suggestion; in particular, generalized Petersen graphs have been further generalized in several different ways. The $I$\emph{-graphs} from the Foster census \cite{FosterCensus} contain the generalized Petersen graphs, and the $I$-graphs are a subset of the $GI$\emph{-graphs}, introduced and studied by Conder, Pisanski, and Zitnik \cite{GIGraphs}. Meanwhile, Sarazin, Pacco, and Previtali ``supergeneralized'' the generalized Petersen graphs in a different way entirely \cite{SupergeneralizedPetersenGraphs}. In some ways, Wilson's \emph{rose window graphs} \cite{RoseWindowGraphs} can be thought of as a tetravalent analogue to generalized Petersen graphs as well.

We propose initially a generalization similar to that of $I$-graphs over generalized Petersen graphs; namely, to introduce a new parameter $a$ so that $\{A_{i},A_{i+a}\}$ is an edge of the graph, rather than $\{A_{i},A_{i+1}\}$. If we denote the generalized propeller graph by $GPr_{n}\left(a,b,c,d\right)$, then $GPr_{n}\left(1,b,c,d\right)$ is the propeller graph $\Pr_{n}\left(b,c,d\right)$. Though we have not studied this new family in great detail, we were able to obtain that $GPr_{10}\left(2,3,1,4\right)$ is arc-transitive and not a propeller graph, as introduced here. No others have yet be found, and it would be interesting to discover whether any others exist at all.

A theorem by Conway prompts a third avenue of investigation. A cycle of a graph is declared \emph{consistent} whenever there exists an automorphism of the graph which acts on the cycle as a one-step rotation. This Biggs-Conway theorem (so-called because it seems that the only record of this talk can be found in a paper by Biggs \cite{BiggsConway}) states that a $d$-valent arc-transitive graph admits $d-1$ orbits of consistent cycles. Given that we have strong restrictions on the parameters of arc-transitive propeller graphs, it may be possible to thoroughly determine the consistent cycle orbits admitted by members of each family.

\appendix
\section{Appendix}
Produced below is a table displaying the $48$ distinct classes of $6$-cycles in a propeller graph.
\begin{longtable}{| p{.20in} | p{1.7in} | p{.25in} | p{.25in} | p{.25in} | p{.25in} | p{.25in} | p{.25in} |}
\hline
\multicolumn{1}{|l|}{No.} & \multicolumn{1}{|l|}{Relations} & \multicolumn{1}{|l|}{$q(\mathcal{X})$} & \multicolumn{1}{|l|}{$r(\mathcal{X})$} & \multicolumn{1}{|l|}{$s(\mathcal{X})$} & \multicolumn{1}{|l|}{$t(\mathcal{X})$} & \multicolumn{1}{|l|}{$u(\mathcal{X})$} & \multicolumn{1}{|l|}{$v(\mathcal{X})$}\\
\hline
$1$ & None ($A$-canonical) & $2$ & $2$ & $2$ & $0$ & $0$ & $0$\\
$2$ & None ($C$-canonical) & $0$ & $0$ & $0$ & $2$ & $2$ & $2$\\
$3$ & $n = 6$ & $1$ & $0$ & $0$ & $0$ & $0$ & $0$\\
$4$ & $b\EQ -4$ & $4$ & $1$ & $1$ & $0$ & $0$ & $0$\\
$5$ & $b\EQ 4$ & $4$ & $1$ & $1$ & $0$ & $0$ & $0$\\
$6$ & $2b\EQ -2$ & $3$ & $3$ & $3$ & $0$ & $0$ & $0$\\
$7$ & $2b\EQ 2$ & $3$ & $3$ & $3$ & $0$ & $0$ & $0$\\
$8$ & $d\EQ -1$ & $4$ & $4$ & $4$ & $4$ & $4$ & $4$\\
$9$ & $d\EQ 1$ & $4$ & $4$ & $4$ & $4$ & $4$ & $4$\\
$10$ & $1+b+d\EQ 0$ & $2$ & $2$ & $2$ & $2$ & $2$ & $2$\\
$11$ & $1+b\EQ d$ & $2$ & $2$ & $2$ & $2$ & $2$ & $2$\\
$12$ & $1+d\EQ b$ & $2$ & $2$ & $2$ & $2$ & $2$ & $2$\\
$13$ & $1\EQ b+d$ & $2$ & $2$ & $2$ & $2$ & $2$ & $2$\\
$14$ & $1+c+d\EQ 0$ & $2$ & $2$ & $2$ & $2$ & $2$ & $2$\\
$15$ & $1+c\EQ d$ & $2$ & $2$ & $2$ & $2$ & $2$ & $2$\\
$16$ & $1+d\EQ c$ & $2$ & $2$ & $2$ & $2$ & $2$ & $2$\\
$17$ & $1\EQ c+d$ & $2$ & $2$ & $2$ & $2$ & $2$ & $2$\\
$18$ & $1+b+c+d\EQ 0$ & $1$ & $1$ & $1$ & $1$ & $1$ & $1$\\
$19$ & $1+b+c\EQ d$ & $1$ & $1$ & $1$ & $1$ & $1$ & $1$\\
$20$ & $1+b+d\EQ c$ & $1$ & $1$ & $1$ & $1$ & $1$ & $1$\\
$21$ & $1+c+d\EQ b$ & $1$ & $1$ & $1$ & $1$ & $1$ & $1$\\
$22$ & $1+b\EQ c+d$ & $1$ & $1$ & $1$ & $1$ & $1$ & $1$\\
$23$ & $1+c\EQ b+d$ & $1$ & $1$ & $1$ & $1$ & $1$ & $1$\\
$24$ & $1+d\EQ b+c$ & $1$ & $1$ & $1$ & $1$ & $1$ & $1$\\
$25$ & $1\EQ b+c+d$ & $1$ & $1$ & $1$ & $1$ & $1$ & $1$\\
$26$ & $3b\EQ 0$ & $0$ & $1$ & $1$ & $0$ & $0$ & $0$\\
$27$ & $3c\EQ 0$ & $0$ & $0$ & $0$ & $1$ & $1$ & $0$\\
$28$ & $2b\EQ -c$ & $0$ & $2$ & $2$ & $1$ & $1$ & $0$\\
$29$ & $2b\EQ c$ & $0$ & $2$ & $2$ & $1$ & $1$ & $0$\\
$30$ & $b\EQ 2c$ & $0$ & $1$ & $1$ & $2$ & $2$ & $0$\\
$31$ & $b\EQ -2c$ & $0$ & $1$ & $1$ & $2$ & $2$ & $0$\\
$32$ & $b\EQ -2d$ & $0$ & $2$ & $2$ & $2$ & $2$ & $4$\\
$33$ & $b\EQ 2d$ & $0$ & $2$ & $2$ & $2$ & $2$ & $4$\\
$34$ & $c\EQ -2$ & $4$ & $2$ & $2$ & $2$ & $2$ & $0$\\
$35$ & $c\EQ 2$ & $4$ & $2$ & $2$ & $2$ & $2$ & $0$\\
$36$ & $b+c+2d\EQ 0$ & $0$ & $1$ & $1$ & $1$ & $1$ & $2$\\
$37$ & $b+c\EQ 2d$ & $0$ & $1$ & $1$ & $1$ & $1$ & $2$\\
$38$ & $b+2d\EQ c$ & $0$ & $1$ & $1$ & $1$ & $1$ & $2$\\
$39$ & $b\EQ c+2d$ & $0$ & $1$ & $1$ & $1$ & $1$ & $2$\\
$40$ & $2+b+c\EQ 0$ & $2$ & $1$ & $1$ & $1$ & $1$ & $0$\\
$41$ & $2+b\EQ c$ & $2$ & $1$ & $1$ & $1$ & $1$ & $0$\\
$42$ & $2+c\EQ b$ & $2$ & $1$ & $1$ & $1$ & $1$ & $0$\\
$43$ & $2\EQ b+c$ & $2$ & $1$ & $1$ & $1$ & $1$ & $0$\\
$44$ & $2c\EQ 2d$ & $0$ & $0$ & $0$ & $3$ & $3$ & $3$\\
$45$ & $2c\EQ -2d$ & $0$ & $0$ & $0$ & $3$ & $3$ & $3$\\
$46$ & $c\EQ 4d$ & $0$ & $0$ & $0$ & $1$ & $1$ & $4$\\
$47$ & $c\EQ -4d$ & $0$ & $0$ & $0$ & $1$ & $1$ & $4$\\
$48$ & $6d\EQ 0$ & $0$ & $0$ & $0$ & $0$ & $0$ & $1$\\
\hline
\caption{The $48$ relations for all possible $6$-cycles.}
\label{Tab:6CycRels}
\end{longtable}

\bibliographystyle{model1-num-names}
\bibliography{ETPropellers.bib}

\end{document}